\documentclass[leqno,12pt]{amsart}
\usepackage{amssymb}
\usepackage{amsmath}
\usepackage{enumerate}
\usepackage{amsfonts}
\usepackage{hyperref}

\usepackage[headheight=18pt, top=25mm, bottom=25mm, left=25mm, right=25mm]{geometry}

\usepackage{tikz}

\definecolor{darkgreen}{RGB}{45, 119, 75}

\newcommand{\supp}{\text{supp }}

\newtheorem{theorem}{Theorem}[section]
\newtheorem{corollary}[theorem]{Corollary}
\newtheorem{lemma}[theorem]{Lemma}
\newtheorem{proposition}[theorem]{Proposition}
\newtheorem{remark}[theorem]{Remark}

\numberwithin{equation}{section}

\hypersetup{
    colorlinks=true,
    linkcolor= blue,
    citecolor =cyan,
    urlcolor = teal,
}

\begin{document}

\title[Generalized translations]{Remarks on Dunkl translations of non-radial kernels}


\author[Jacek Dziubański]{Jacek Dziubański}
\author[Agnieszka Hejna]{Agnieszka Hejna}
\begin{abstract} On $
\mathbb R^N$ equipped with a root system $R$ and a multiplicity function $k>0$, we study the generalized (Dunkl) translations $\tau_{\mathbf x}g(-\mathbf y)$ of not necessarily radial kernels $g$. Under certain regularity assumptions on $g$, we  derive bounds for $\tau_{\mathbf x}g(-\mathbf y)$ by means the Euclidean distance $\|\mathbf x-\mathbf y\|$ and the distance $d(\mathbf x,\mathbf y)=\min_{\sigma \in G} \| \mathbf x-\sigma (\mathbf y)\|$, where $G$ is the reflection group associated with $R$. Moreover, we prove that $\tau$ does not preserve positivity, that is, there is a non-negative Schwartz class function $\varphi$, such that $\tau_{\mathbf x}\varphi (-\mathbf y)<0$ for some points $\mathbf x,\mathbf y\in\mathbb R^N$ 
\end{abstract}

\address{Jacek Dziubański, Uniwersytet Wroc\l awski,
Instytut Matematyczny,
Pl. Grunwaldzki 2,
50-384 Wroc\l aw,
Poland}
\email{jdziuban@math.uni.wroc.pl}

\address{Agnieszka Hejna, Uniwersytet Wroc\l awski,
Instytut Matematyczny,
Pl. Grunwaldzki 2,
50-384 Wroc\l aw,
Poland 
\&
Department of Mathematics,
Rutgers University,
Piscataway, NJ 08854-8019, USA}
\email{hejna@math.uni.wroc.pl}

\maketitle

\section{Introduction }\label{Sec:Intro}

{We consider $\mathbb R^N$ equipped with a root system $R$ and a multiplicity function $k>0$. 
Behavior of the generalized  Dunkl translations $\tau_{\mathbf{x}}g(-\mathbf{y})$ and, consequently, boundedness of the generalized  convolution operators  
 \begin{align*}
  f \longmapsto f*g(\mathbf x) =\int_{\mathbb{R}^N} f(\mathbf{y}) \tau_{\mathbf{x}}g(-\mathbf{y})\, dw(\mathbf{y}),    
 \end{align*}
on various function spaces  are ones of the main problems in the harmonic analysis  in the Dunkl setting.}  Here and subsequently, $dw$ is the measure associated with the system {$(R,k)$} (see \eqref{eq:measure}).  If $f\in L^p(dw)$, $g\in L^1(dw)$ and one of them is radial then, thanks to the R\"osler formula (see~\eqref{eq:translation-radial}) on  translations of radial functions, one has
 \begin{equation}\label{eq:convolution_bounds1}
\| f*g\|_{L^p(dw)}\leq C \| f\|_{L^p(dw)}\| g\|_{L^1(dw)}
 \end{equation}
 with $C=1$.
Further, since the generalized translations of any radial non-negative function $g$ are non-negative, some pointwise estimates for $\tau_{\mathbf{x}}g(-\mathbf{y})$ can be derived from the bounds of the heat kernel $h_t(\mathbf{x},\mathbf{y})$ (see Proposition~\ref{prop:tran_radiall}).  In particular,  if $g$ is a radial function such that $|g(\mathbf{x})|\leq C_M(1+\|\mathbf{x}\|)^{-M}$ for all  $M>0$, then for any $M'>0$,
\begin{equation}\label{eq:trans-Euclidean}
|\tau_{\mathbf{x}}g(-\mathbf{y})| \leq C'_M { w(B(\mathbf x, 1))^{-1}}(1+\|\mathbf{x}-\mathbf{y}\|)^{-2} (1+d(\mathbf{x},\mathbf{y}))^{-M'},
\end{equation}
{where $d(\mathbf{x},\mathbf{y})=\min_{\sigma\in G} \|\mathbf x-\sigma(\mathbf y)\|$, $G$ is the reflection group associated with $R$  (see~\eqref{eq:d}).}

On the other hand, the $L^p(dw)$-bounds for the generalized translations $\tau_{\mathbf{x}}g$ of non-radial $L^p$-functions for $p\ne 2$ is an open problem as well as the inequality
 \eqref{eq:convolution_bounds1}.
However, if we assume some regularity of a (non-radial) function $g$ in its smoothness and decay, then 
\begin{equation}\label{eq:simple-est}
    |\tau_{\mathbf{x}} g(-\mathbf{y})|\leq C w(B(\mathbf x,1))^{-1} (1+d(\mathbf x,\mathbf y))^{-M},
\end{equation}
(see \cite[Proposition 5.1]{DzH-square}) and, consequently,
\begin{equation}
\| f*g\|_{L^p(dw)}\leq C \| f\|_{L^p(dw)}.
\end{equation}
 The estimates of the form \eqref{eq:simple-est}, which make use of the distance  $d(\mathbf{x},\mathbf{y})$ of the orbits and the measures of the balls, seem to be useful, because they  allow one to reduce some problems to the setting of spaces of homogeneous type and apply tools from the theory of these spaces for obtaining some analytic-spirit results. For instance, in~\cite{ADzH} this approach was used in order to define and characterize the real Dunkl Hardy space $H^1_{\Delta_k}$ by means of boundary values  of the Dunkl conjugate harmonic functions, maximal functions associated with radial kernels, the relevant Riesz transforms, square functions and atoms (which were defined in the spirit of \cite{Hof}).  From the point of view of non-radial kernels $g$, in some cases, the estimates \eqref{eq:simple-est} can be used as a substitute for $L^p$-boundedness of the Dunkl translations (see~\cite{DzH}). 

On the other hand, it was noticed that in some cases the estimates of the form \eqref{eq:simple-est} are not strong enough to obtain some harmonic analysis spirit results involving Dunkl translations and convolutions. For example in order to prove that the Hardy space $H^1_{\Delta_k}$ admits atomic decomposition into Coifman-Weiss atoms, the authors of \cite{DzH1} needed the following estimates for the generalized translations of radial continuous functions supported in the unit ball:
\begin{equation}\label{eq:weaker}
     |\tau_{\mathbf{x}} g(-\mathbf{y})|\leq C w(B(\mathbf x,1))^{-1} (1+\| \mathbf x-\mathbf y\|)^{-1}\chi_{[0,1]}(d(\mathbf x,\mathbf y). 
\end{equation}
The estimate \eqref{eq:weaker} is a
 slightly weaker version of  \eqref{eq:trans-Euclidean} because  the factor $(1+\|\mathbf x-\mathbf y\|)$ is raised to the power negative one, however its presence  is crucial for the proof. 
Further, a presence of 
the factor  $(1+\|\mathbf x-\mathbf y\|)^{-\delta}$ (or its scaled version) in estimates of integral kernels helps to handle harmonic analysis problems in the Dunkl setting (see e.g. \cite[Section 5]{singular} and \cite{Tan} for a study of singular integrals). 

Another question can be asked for the exponent(s) associated with the Euclidean distance(s) in estimates of generalized translations of $g$. It was proved  in \cite{DH-heat} that for the Dunkl heat kernel $h_t(\mathbf x,\mathbf y)$ the exponents  depend on sequences of reflections needed to move $\mathbf y$ to a Weyl chamber of $\mathbf x$. To be more precise, the following   upper and lower bounds for $h_t(\mathbf x,\mathbf y)$ hold: 
for all $c_l>1/4$ and $0<c_u<1/4$ there are constants $C_l,C_u>0$ such that  
\begin{equation}\label{eq:DH-heat}
     C_{l}w(B(\mathbf{x},\sqrt{t}))^{-1}e^{-c_{l}\frac{d(\mathbf{x},\mathbf{y})^2}{t}} \Lambda(\mathbf x,\mathbf y,t) \leq    h_t(\mathbf{x},\mathbf{y}) \leq  C_{u}w(B(\mathbf{x},\sqrt{t}))^{-1}e^{-c_{u}\frac{d(\mathbf{x},\mathbf{y})^2}{t}} \Lambda(\mathbf x,\mathbf y,t),
\end{equation}
where $\Lambda(\mathbf x,\mathbf y,t)$ can be expressed by means of  some   rational functions of  $\| \mathbf x-\sigma(\mathbf y)\|/\sqrt{t}$ (see Theorem~\ref{teo:1} for details). 
The estimate \eqref{eq:DH-heat} improves the {known} bound 
\begin{equation}\label{eq:intro_heat_2}
    h_t(\mathbf{x},\mathbf{y}) \lesssim \left(1+\frac{\|\mathbf{x}-\mathbf{y}\|^2}{t}\right)^{-1}\frac{1}{\max(w(B(\mathbf{x},\sqrt{t})),w(B(\mathbf{y},\sqrt{t})))}e^{-\frac{cd(\mathbf{x},\mathbf{y})^2}{t}}
\end{equation}
 (see \cite[Theorem 3.1]{DzH1} for a proof of \eqref{eq:intro_heat_2}), which can be used, as we remarked out,  for proving estimates for translations of radial kernels. An alternative proof of \eqref{eq:intro_heat_2} which uses a Poincar\'e inequality was announced by W. Hebisch. Let us also point out the presence of the same function $\Lambda$ in the upper and lower bounds \eqref{eq:DH-heat}. Thus if $d(\mathbf x,\mathbf y)^2\leq t$, the estimates \eqref{eq:DH-heat} are sharp. 

The goal of this paper is to present some properties of the generalized translations $\tau_{\mathbf{x}} g(-\mathbf{y})$ of non-radial  kernels $g$, and, in particular {propose some methods which allow to} derive estimates for $\tau_{\mathbf{x}} g(-\mathbf{y})$ {and express them} in terms of various distances and measures $w(B)$ of appropriate balls. {We prove that if a (non-radial) function $g$ is sufficiently regular, then 
\begin{equation}
     |\tau_{\mathbf{x}} g(-\mathbf{y})|\leq C w(B(\mathbf x,d(\mathbf x,\mathbf y)+1))^{-1} (1+\|\mathbf x-\mathbf y\|)^{-1}(1+d(\mathbf x,\mathbf y))^{-M}
\end{equation}
(see Theorem \ref{teo:trans-g}). 
Further  we aim to obtain estimates for $\tau_{\mathbf{x}} g(-\mathbf{y})$ for non-radial $g$ and interpret them in the context of~\eqref{eq:DH-heat}. From one point of view, one can expect the upper estimates making use of the same function $\Lambda(\mathbf{x},\mathbf{y},t)$. Since  in the case of non-radial kernels the R\"osler's formula is not available,  we need a different approach, which is presented in Section~\ref{sec:Estimates}, for obtaining estimates for the generalized translations of non-radial Schwartz-class functions which involve  the function $\Lambda^{1/2}$ (see Theorem~\ref{teo:Schwartz-transl}). Then we use the same results in order to unify two approaches to the theory of singular integrals from~\cite{singular} and~\cite{Tan} (see Section~\ref{sub:conv} and Theorem~\ref{teo:is_CZ}). Further, it turns out that our approach developed in Section~\ref{sec:Estimates} can be used in order to prove non-positivity of the Dunkl translations operators for any root system $R \neq \emptyset$ (see Theorem~\ref{teo:negative} for details). }

\section{Preliminaries and notation}

\subsection{Dunkl theory}

In this section we present basic facts concerning the theory of the Dunkl operators.   For more details we refer the reader to~\cite{Dunkl},~\cite{Roesle99},~\cite{Roesler3}, and~\cite{Roesler-Voit}. 

We consider the Euclidean space $\mathbb R^N$ with the scalar product $\langle \mathbf{x},\mathbf y\rangle=\sum_{j=1}^N x_jy_j
$, where $\mathbf x=(x_1,...,x_N)$, $\mathbf y=(y_1,...,y_N)$, and the norm $\| \mathbf x\|^2=\langle \mathbf x,\mathbf x\rangle$.

A {\it normalized root system}  in $\mathbb R^N$ is a finite set  $R\subset \mathbb R^N\setminus\{0\}$ such that $R \cap \alpha \mathbb{R} = \{\pm \alpha\}$,  $\sigma_\alpha (R)=R$, and $\|\alpha\|=\sqrt{2}$ for all $\alpha\in R$, where $\sigma_\alpha$ is defined by 

\begin{equation}\label{reflection}\sigma_\alpha (\mathbf x)=\mathbf x-2\frac{\langle \mathbf x,\alpha\rangle}{\|\alpha\|^2} \alpha.
\end{equation}

The finite group $G$ generated by the reflections $\sigma_{\alpha}$, $\alpha \in R$, is called the {\it Coxeter group} ({\it reflection group}) of the root system.

A~{\textit{multiplicity function}} is a $G$-invariant function $k:R\to\mathbb C$ which will be fixed and $> 0$  throughout this paper.  

The associated measure $dw$ is defined by $dw(\mathbf x)=w(\mathbf x)\, d\mathbf x$, where 
 \begin{equation}\label{eq:measure}
w(\mathbf x)=\prod_{\alpha\in R}|\langle \mathbf x,\alpha\rangle|^{k(\alpha)}.
\end{equation}
Let $\mathbf{N}=N+\sum_{\alpha \in R}k(\alpha)$. Then, 
\begin{equation}\label{eq:t_ball} w(B(t\mathbf x, tr))=t^{\mathbf{N}}w(B(\mathbf x,r)) \ \ \text{\rm for all } \mathbf x\in\mathbb R^N, \ t,r>0, 
\end{equation}
{where, here and subsequently, $B(\mathbf x,r)=\{\mathbf y\in\mathbb R^N: \|\mathbf x-\mathbf y\|\leq r\}$.} Observe that there is a constant $C>0$ such that for all $\mathbf{x} \in \mathbb{R}^N$ and $r>0$ we have
\begin{equation}\label{eq:balls_asymp}
C^{-1}w(B(\mathbf x,r))\leq  r^{N}\prod_{\alpha \in R} (|\langle \mathbf x,\alpha\rangle |+r)^{k(\alpha)}\leq C w(B(\mathbf x,r)),
\end{equation}
so $dw(\mathbf x)$ is doubling, that is, there is a constant $C>0$ such that
\begin{equation}\label{eq:doubling} w(B(\mathbf x,2r))\leq C w(B(\mathbf x,r)) \ \ \text{ for all } \mathbf x\in\mathbb R^N, \ r>0.
\end{equation}
Moreover, there exists a constant $C\ge1$ such that,
for every $\mathbf{x}\in\mathbb{R}^N$ and for all $r_2\ge r_1>0$,
\begin{equation}\label{eq:growth}
C^{-1}\Big(\frac{r_2}{r_1}\Big)^{N}\leq\frac{{w}(B(\mathbf{x},r_2))}{{w}(B(\mathbf{x},r_1))}\leq C \Big(\frac{r_2}{r_1}\Big)^{\mathbf{N}}.
\end{equation}

For $\xi \in \mathbb{R}^N$, the {\it Dunkl operators} $T_\xi$  are the following $k$-deformations of the directional derivatives $\partial_\xi$ by   difference operators:
\begin{equation}\label{eq:T_xi}
     T_\xi f(\mathbf x)= \partial_\xi f(\mathbf x) + \sum_{\alpha\in R} \frac{k(\alpha)}{2}\langle\alpha ,\xi\rangle\frac{f(\mathbf x)-f(\sigma_\alpha(\mathbf{x}))}{\langle \alpha,\mathbf x\rangle}.
\end{equation}
The Dunkl operators $T_{\xi}$, which were introduced in~\cite{Dunkl}, commute and are skew-symmetric with respect to the $G$-invariant measure $dw$, i.e. for reasonable functions $f,g$ (for instance, $f,g \in \mathcal{S}(\mathbb{R}^N)$) we have
\begin{equation}\label{eq:by_parts}
    \int_{\mathbb{R}^N}T_{\xi}f(\mathbf{x})g(\mathbf{x})\,dw(\mathbf{x})=-\int_{\mathbb{R}^N}f(\mathbf{x})T_{\xi}g(\mathbf{x})\,dw(\mathbf{x}).
\end{equation}
Let us denote $T_j=T_{e_j}$, where $\{e_j\}_{1 \leq j \leq N}$ is a canonical orthonormal  basis of $\mathbb{R}^N$.

For $f,g \in C^1(\mathbb{R}^N)$, we have the following Leibniz-type rule
\begin{equation}\label{eq:general_Leibniz}
    T_j(fg)(\mathbf{x})=(T_jf)(\mathbf{x})g(\mathbf{x})+f(\mathbf{x})\partial_{j}g(\mathbf{x})+\sum_{\alpha \in R}\frac{k(\alpha)}{2} \langle\alpha,e_j\rangle f(\sigma_{\alpha}(\mathbf{x}))\frac{g(\mathbf{x})-g(\sigma_{\alpha}(\mathbf{x}))}{\langle \mathbf{x}, \alpha \rangle}.
\end{equation}
For multi-index $\mathbf{\beta}=(\beta_1,\beta_2,\ldots,\beta_N) \in \mathbb N_0^N=(\mathbb{N} \cup \{0\})^{N} $, we denote 
\begin{equation}\label{eq:multi-index_len}
|\beta|=\beta_1+\ldots +\beta_N, \ \ T_j^{\mathbf{0}}={\rm id}, \ \ \partial^{\beta}=\partial_1^{\beta_1} \circ \ldots \circ \partial_N^{\beta_N}, \ \ T^{\beta}=T_{1}^{\beta_1} \circ \ldots \circ T_{N}^{\beta_N}.
\end{equation}
For fixed $\mathbf y\in\mathbb R^N$, the {\it Dunkl kernel} $\mathbf{x} \mapsto E(\mathbf x,\mathbf y)$ is the unique analytic solution to the system
\begin{equation}\label{eq:E_def}
    T_\xi f=\langle \xi,\mathbf y\rangle f, \ \ f(0)=1.
\end{equation}
The function $E(\mathbf x ,\mathbf y)$, which generalizes the exponential  function $e^{\langle \mathbf x,\mathbf y\rangle}$, has a  unique extension to a holomorphic function on $\mathbb C^N\times \mathbb C^N$. It was proved in~\cite[Corollary 5.3]{Roesle99} that for all $\mathbf{x},\mathbf{y} \in \mathbb{R}^N$ and $\nu \in \mathbb{N}_0^{N}$ we have
\begin{equation}\label{eq:E_est}
    |\partial^{\nu}_{\mathbf{y}}E(\mathbf{x},i\mathbf{y})| \leq \|\mathbf{x}\|^{|\nu|}.
\end{equation}

\subsection{Dunkl transform}

Let $f \in L^1(dw)$. We define the \textit{Dunkl transform }$\mathcal{F}f$ of $f$ by
\begin{equation}\label{eq:Dunkl_transform}
    \mathcal{F} f(\xi)=\mathbf{c}_k^{-1}\int_{\mathbb{R}^N}f(\mathbf{x})E(\mathbf{x},-i\xi)\, {dw}(\mathbf{x}),
\end{equation}
where
\begin{equation*}
    \mathbf{c}_k=\int_{\mathbb{R}^N}e^{-\frac{{\|}\mathbf{x}{\|}^2}2}\,{dw}(\mathbf{x}){>0}
\end{equation*}
is so called \textit{Mehta–Macdonald integral}. The Dunkl transform is a generalization of the Fourier transform on $\mathbb{R}^N$. It was introduced in~\cite{D5} for $k \geq 0$ and further studied in~\cite{dJ1} in the more general context.  It was proved in~\cite[Corollary 2.7]{D5} (see also~\cite[Theorem 4.26]{dJ1}) that it extends uniquely  to  an isometry on $L^2(dw)$, i.e.,
   \begin{equation}\label{eq:Plancherel}
       \|f\|_{L^2(dw)}=\|\mathcal{F}f\|_{L^2(dw)} \text{ for all }f \in L^2(dw)\cap L^1(dw).
   \end{equation}
We have also the following inversion theorem (\cite[Theorem 4.20]{dJ1}): for all $f \in L^1(dw)$ such that $\mathcal{F}f \in L^1(dw)$ we have
\begin{equation}\label{eq:inverse}
    f(\mathbf{x})=(\mathcal{F})^2f(-\mathbf{x}) \text{ for almost all }\mathbf{x} \in \mathbb{R}^N\textup{.}
\end{equation}
The inverse $\mathcal F^{-1}$ of $\mathcal{F}$ has the form
  \begin{equation}\label{eq:inverse_teo} \mathcal F^{-1} f(\mathbf{x})=\mathbf{c}_k^{-1}\int_{\mathbb R^N} f(\xi)E(i\xi, \mathbf x)\, dw(\xi)= \mathcal{F}f(-\mathbf{x})\quad \text{\rm for } f\in L^1(dw)\textup{.}
  \end{equation}
It can be easily checked using~\eqref{eq:E_def} that for compactly supported $f \in L^1(dw)$ we have
\begin{equation}\label{eq:T_j_fourier_side}
    T_{j}(\mathcal{F}f)(\xi)=\mathcal{F}g(\xi), \text{ where }g(\mathbf x)=-i x_{j}f(\mathbf x).
\end{equation}

\subsection{Dunkl translations}
Suppose that $f \in \mathcal{S}(\mathbb{R}^N)$ {
(the Schwartz class of functions on $\mathbb R^N)$)} and $\mathbf{x} \in \mathbb{R}^N$. We define the \textit{Dunkl translation }$\tau_{\mathbf{x}}f$ of $f$ to be
\begin{equation}\label{eq:translation}
    \tau_{\mathbf{x}} f(-\mathbf{y})=\mathbf{c}_k^{-1} \int_{\mathbb{R}^N}{E}(i\xi,\mathbf{x})\,{E}(-i\xi,\mathbf{y})\,\mathcal{F}f(\xi)\,{dw}(\xi)=\mathcal{F}^{-1}(E(i\cdot,\mathbf{x})\mathcal{F}f)(-\mathbf{y}).
\end{equation}
The Dunkl translation was introduced in~\cite{R1998}. The definition can be extended to the functions which are not necessary in $\mathcal{S}(\mathbb{R}^N)$. For instance, thanks to the Plancherel's theorem (see~\eqref{eq:Plancherel}), one can define the Dunkl translation of $L^2(dw)$ function $f$ by
\begin{equation}\label{eq:translation_Fourier}
    \tau_{\mathbf{x}}f(-\mathbf{y})=\mathcal{F}^{-1}(E(i\cdot,\mathbf{x})\mathcal{F}f(\cdot))(-\mathbf{y})
\end{equation}
(see~\cite{R1998} of~\cite[Definition 3.1]{ThangaveluXu}). In particular, it follows from~\eqref{eq:translation_Fourier},~\eqref{eq:E_est}, and~\eqref{eq:Plancherel} that for all $\mathbf{x} \in \mathbb{R}^N$ the operators $f \mapsto \tau_{\mathbf{x}}f$ are contractions on $L^2(dw)$. Here and subsequently, we  write $g(\mathbf x,\mathbf y):=\tau_{\mathbf x}g(-\mathbf y)$.

We will need the following result concerning the support of the Dunkl translated compactly supported function.

\begin{theorem}[\cite{DzH} Theorem 1.7]\label{teo:support}
 Let $f \in L^2(dw)$, $\text{\rm supp}\, f \subseteq B(0,r)$, and $\mathbf{x} \in \mathbb{R}^N$. Then  
\begin{equation}\label{eq:inclusion_support}
\supp \tau_{\mathbf{x}}f(-\, \cdot) \subseteq \mathcal{O}(B(\mathbf x,r)).
\end{equation}
\end{theorem}
Here and subsequently, for a measurable set $A \subseteq \mathbb{R}^N$ we denote
\begin{equation*}
    \mathcal{O}(A)=\{\sigma(\mathbf{z})\;:\; \sigma \in G,\, \mathbf{z} \in A\}.
\end{equation*}

\subsection{Dunkl translations of radial functions}
  The following specific formula was obtained by R\"osler \cite{Roesler2003} for the Dunkl translations of (reasonable) radial functions $f({\mathbf{x}})=\tilde{f}({\|\mathbf{x}\|})$:
\begin{equation}\label{eq:translation-radial}
\tau_{\mathbf{x}}f(-\mathbf{y})=\int_{\mathbb{R}^N}{(\tilde{f}\circ A)}(\mathbf{x},\mathbf{y},\eta)\,d\mu_{\mathbf{x}}(\eta)\text{ for all }\mathbf{x},\mathbf{y}\in\mathbb{R}^N.
\end{equation}
Here
\begin{equation*}
A(\mathbf{x},\mathbf{y},\eta)=\sqrt{{\|}\mathbf{x}{\|}^2+{\|}\mathbf{y}{\|}^2-2\langle \mathbf{y},\eta\rangle}=\sqrt{{\|}\mathbf{x}{\|}^2-{\|}\eta{\|}^2+{\|}\mathbf{y}-\eta{\|}^2}
\end{equation*}
and $\mu_{\mathbf{x}}$ is a probability measure, 
which is supported in the set $\operatorname{conv}\mathcal{O}(\mathbf{x})$  (the convex hull of the orbit of $\mathbf x$ under {the action of $G$}).

\subsection{Dunkl convolution}
Assume that $f,g \in L^2(dw)$. The \textit{generalized convolution} (or \textit{Dunkl convolution}) $f*g$ is defined by the formula
\begin{equation}\label{eq:conv_transform}
    f*g(\mathbf{x})=\mathbf{c}_k\mathcal{F}^{-1}\big((\mathcal{F}f)(\mathcal{F}g)\big)(\mathbf{x}),
\end{equation}
equivalently, by
\begin{equation}\label{eq:conv_translation}
    (f*g)(\mathbf{x})=\int_{\mathbb{R}^N}f(\mathbf{y})\,\tau_{\mathbf{x}}g(-\mathbf{y})\,{dw}(\mathbf{y})=\int_{\mathbb{R}^N}g(\mathbf{y})\,\tau_{\mathbf{x}}f(-\mathbf{y})\,{dw}(\mathbf{y}).
\end{equation}
Generalized convolution of $f,g \in \mathcal{S}(\mathbb{R}^N)$ was considered in~\cite{R1998} and~\cite{Trimeche}, the definition was extended to $f,g \in L^2(dw)$ in~\cite{ThangaveluXu}. 

\subsection{Generalized heat semigroup and heat kernel}

The {\it Dunkl Laplacian} associated with $R$ and $k$  is the differential-difference operator $\Delta_k=\sum_{j=1}^N T_{j}^2$, which  acts on $C^2(\mathbb{R}^N)$-functions by\begin{align*}
    \Delta_k f(\mathbf x)=\Delta_{\rm eucl} f(\mathbf x)+\sum_{\alpha\in R} k(\alpha) \delta_\alpha f(\mathbf x),\ \ \delta_\alpha f(\mathbf x)=\frac{\partial_\alpha f(\mathbf x)}{\langle \alpha , \mathbf x\rangle} - \frac{\|\alpha\|^2}{2} \frac{f(\mathbf x)-f(\sigma_\alpha (\mathbf x))}{\langle \alpha, \mathbf x\rangle^2}.
\end{align*}

The operator $\Delta_k$ is essentially self-adjoint on $L^2(dw)$ (see for instance \cite[Theorem\;3.1]{AH}) and generates a semigroup $H_t$  of linear self-adjoint contractions on $L^2(dw)$. The semigroup has the form
  \begin{align*}
  H_t f(\mathbf x)=\int_{\mathbb R^N} h_t(\mathbf x,\mathbf y)f(\mathbf y)\, dw(\mathbf y),
  \end{align*}
  where the heat kernel
  \begin{equation}\label{eq:heat_formula}
      h_t(\mathbf x,\mathbf y)={\boldsymbol c}_k^{-1} (2t)^{-\mathbf{N}/2}E\Big(\frac{\mathbf x}{\sqrt{2t}}, \frac{\mathbf y}{\sqrt{2t}}\Big)e^{-(\| \mathbf x\|^2+\|\mathbf y\|^2)\slash (4t)}
  \end{equation}
  is a $C^\infty$-function of all the variables $\mathbf x,\mathbf y \in \mathbb{R}^N$, $t>0$, and satisfies \begin{equation}\label{eq:h_positive}
  0<h_t(\mathbf x,\mathbf y)=h_t(\mathbf y,\mathbf x).
  \end{equation}
  In terms of the  generalized translations we have 
  \begin{equation}\label{eq:ht_by_translation}
      h_t(\mathbf{x},\mathbf{y} )
=\tau_{\mathbf{x}}h_t(-\mathbf{y}), \text{ where } h_t(\mathbf{x})=\tilde{h}_t(\|\mathbf{x}\|)
=\mathbf{c}_k^{-1}\,(2t)^{-\mathbf{N}/2}\,e^{-\frac{{\|}\mathbf{x}{\|}^2}{4t}},
  \end{equation}
and, in terms of the Dunkl transform,
\begin{equation}\label{eq:ht_by_Fourier}
    \mathcal{F}h_{t}(\xi)=\mathbf{c}_k^{-1}e^{-t\|\xi\|^2}.
\end{equation}
\subsection{Upper and lower heat kernel bounds}

The closures of connected components of
\begin{equation*}
    \{\mathbf{x} \in \mathbb{R}^{N}\;:\; \langle \mathbf{x},\alpha\rangle \neq 0 \text{ for all }\alpha \in R\}
\end{equation*}
are called (closed) \textit{Weyl chambers}. 
We define the distance of the orbit of $\mathbf x$ to the orbit of $\mathbf y$ by
\begin{equation}\label{eq:d}
d(\mathbf x,\mathbf y)=\min \{ \| \mathbf x-\sigma (\mathbf y)\|: \sigma \in G\}.    
\end{equation}

For a finite  sequence $\boldsymbol \alpha  =(\alpha_1,\alpha_2,\ldots,\alpha_m)$ of elements of $R$, $\mathbf x,\mathbf y\in \mathbb R^N$ and $t>0$,  let $\ell(\mathbf{\boldsymbol \alpha}):=m$ 
be the length of $\boldsymbol \alpha$, 
\begin{equation}
    \sigma_{\boldsymbol \alpha}:=\sigma_{\alpha_m}\circ \sigma_{\alpha_{m-1}}\circ \ldots\circ\sigma_{\alpha_1}, 
\end{equation}
and 
 \begin{equation}\label{eq:rho}\begin{split}
    &\rho_{\boldsymbol \alpha}(\mathbf{x},\mathbf{y},t)\\&:=\left(1+\frac{\|\mathbf{x}-\mathbf{y}\|}{\sqrt{t}}\right)^{-2}\left(1+\frac{\|\mathbf{x}-\sigma_{\alpha_1}(\mathbf{y})\|}{\sqrt{t}}\right)^{-2}\left(1+\frac{\|\mathbf{x}-\sigma_{\alpha_2} \circ \sigma_{\alpha_1}(\mathbf{y})\|}{\sqrt{t}}\right)^{-2}\times \ldots  \\&\times \left(1+\frac{\|\mathbf{x}-\sigma_{\alpha_{m-1}} \circ \ldots \circ \sigma_{\alpha_1}(\mathbf{y})\|}{\sqrt{t}}\right)^{-2}.
    \end{split}
\end{equation}

For $\mathbf x,\mathbf y\in\mathbb R^N$, let 
    $ n (\mathbf x,\mathbf y)=0$ if $  d(\mathbf x,\mathbf y)=\| \mathbf x-\mathbf y\|$ and  
    
    \begin{equation}\label{eq:n}
        n(\mathbf x,\mathbf y) = 
    \min\{m\in\mathbb Z: d(\mathbf x,\mathbf y)=\| \mathbf x-\sigma_{\alpha_{m}}\circ \ldots \circ \sigma_{\alpha_2}\circ\sigma_{\alpha_1}(\mathbf y)\|,\quad \alpha_j\in R\}     
    \end{equation}
otherwise. In other words, $n(\mathbf x,\mathbf y)$ is the smallest number of reflections $\sigma_\alpha$ which are needed to move $\mathbf y$ to a (closed) Weyl chamber of $\mathbf x$. 
We also allow  $\boldsymbol{\alpha}$ to be the empty sequence, denoted by $\boldsymbol{\alpha} =\emptyset$. Then for $\boldsymbol{\alpha}=\emptyset$, we set:   $\sigma_{\boldsymbol{\alpha}}=\text{\rm id}$ (the identity operator), $\ell(\boldsymbol{\alpha})=0$, and $\rho_{\boldsymbol{\alpha}}(\mathbf{x},\mathbf{y},t)=1$ for all $\mathbf{x},\mathbf{y} \in \mathbb{R}^N$ and $t>0$.

We say that  a finite  sequence $\boldsymbol \alpha  =(\alpha_1,\alpha_2,\ldots,\alpha_m)$ of  roots is {\it admissible for a pair}  $(\mathbf x,\mathbf y)\in\mathbb R^N\times\mathbb R^N$  if $n(\mathbf{x},\sigma_{\boldsymbol{\alpha}}(\mathbf{y}))=0$. In other words, the composition $\sigma_{\alpha_m}\circ\sigma_{\alpha_{m-1}}\circ \ldots\circ \sigma_{\alpha_1}$ of the reflections $\sigma_{\alpha_j}$ maps $\mathbf y$ to the Weyl chamber of $\mathbf x$. 
The set of the all admissible sequences $\boldsymbol \alpha$ for the pair  $(\mathbf x,\mathbf y)$  will be denoted by $\mathcal A(\mathbf x,\mathbf y)$.  
Note that if $n(\mathbf x,\mathbf y)=0$, then $\boldsymbol{\alpha}=\emptyset \in \mathcal A(\mathbf{x},\mathbf{y})$.

Let us define
\begin{equation}\label{eq:Lambda_def}
    \Lambda (\mathbf x,\mathbf y,t):=\sum_{\boldsymbol \alpha \in \mathcal{A}(\mathbf{x},\mathbf{y}), \;\ell (\boldsymbol \alpha) \leq |G|} \rho_{\boldsymbol \alpha}(\mathbf x,\mathbf y,t).
\end{equation}    

The following upper and lower bounds for $h_t(\mathbf x,\mathbf y)$ were proved in \cite{DH-heat}.
\begin{theorem}[{\cite{DH-heat} and \cite{DH-Schrodinger} Remark 2.3}]
\label{teo:1}
Assume that $0<c_{u}<1/4$ and $c_l>1/4$. There are constants $C_{u},C_{l}>0$ such that for all $\mathbf{x},\mathbf{y} \in \mathbb{R}^N$ and $t>0$ we have 
\begin{equation}\label{eq:main_lower}
 C_{l}w(B(\mathbf{x},\sqrt{t}))^{-1}e^{-c_{l}\frac{d(\mathbf{x},\mathbf{y})^2}{t}} \Lambda(\mathbf x,\mathbf y,t) \leq    h_t(\mathbf{x},\mathbf{y}), 
\end{equation}
\begin{equation}\label{eq:main_claim}
    h_t(\mathbf{x},\mathbf{y}) \leq C_{u}w(B(\mathbf{x},\sqrt{t}))^{-1}e^{-c_{u}\frac{d(\mathbf{x},\mathbf{y})^2}{t}} \Lambda(\mathbf x,\mathbf y,t).
\end{equation}
\end{theorem}

We also have the following regularity estimate for $h_t(\mathbf{x},\mathbf{y})$ (\cite[Theorem 6.1]{DH-heat}).

\begin{lemma}
Let $\varepsilon_1 \in (0,1]$. There is a constant $C>0$ such that for all $\mathbf{x},\mathbf{y},\mathbf{y}' \in \mathbb{R}^N$ and $t>0$ we have
\begin{equation}\label{eq:h_t_regular}
    |h_t(\mathbf{x},\mathbf{y})-h_t(\mathbf{x},\mathbf{y}')| \leq C \left(\frac{\|\mathbf{y}-\mathbf{y}'\|}{\sqrt{t}}\right)^{\varepsilon_1}\left(h_{2t}(\mathbf{x},\mathbf{y})+h_{2t}(\mathbf{x},\mathbf{y}')\right).
\end{equation}
\end{lemma}

As an application of Theorem~\ref{teo:1} and 
 \eqref{eq:translation-radial} it is possible to describe the behavior of the measure  $\mu_{\mathbf x}$ near the points   $\sigma(\mathbf{x})$ for $\sigma \in G$. These estimates give another proof of the theorem of Gallardo and Rejeb (see~\cite[Theorem A 3)]{GallardoRejeb}), which says that $\sigma(\mathbf x)$, $\sigma \in G$, belong to the support of the measure $\mu_{\mathbf x}$.

\begin{theorem}[\cite{DH-heat}]\label{teo:rejeb}
For $\mathbf{x} \in \mathbb{R}^N$ and $t>0$ we set
\begin{equation}\label{eq:U}
    U(\mathbf x,t):=\{ \eta\in \text{\rm conv}\, \mathcal O(\mathbf x): \| \mathbf x\|^2-\langle\mathbf x,\eta\rangle \leq t  \}.
\end{equation}
There is a constant $C>0$ such that for all $\mathbf{x} \in \mathbb{R}^N$, $t>0$, and $\sigma \in G$ we have 
\begin{equation}\label{eq:precise_measure}
    C^{-1}\frac{t^{\mathbf N/2}\Lambda(\mathbf x,\sigma(\mathbf x),t)}{w(B(\mathbf x,\sqrt{t}))}\leq  \mu_{\mathbf x}(U(\sigma(\mathbf x),t))\leq C\frac{t^{\mathbf N/2}\Lambda(\mathbf x,\sigma(\mathbf x),t)}{w(B(\mathbf x,\sqrt{t}))}.
\end{equation}
\end{theorem}

\subsection{Kernel of the Dunkl--Bessel potential}
For an even positive integer $s$, we set
\begin{equation}\label{eq:bessel}
    J^{\{s\}}:=\mathcal{F}^{-1}(1+\|\cdot\|^{2})^{-s/2}, \text{ i.e. }\mathcal{F}J^{\{s\}}(\xi)=(1+\|\xi\|^{2})^{-s/2}.
\end{equation}
It can be easily checked that for $\mathbf{x},\mathbf{y} \in \mathbb{R}^N$ we have
\begin{equation}\label{eq:subordination_bessel}
    J^{\{s\}}(\mathbf{x})={\Gamma \Big(\frac{s}{2}\Big)^{-1}}\int_0^{\infty}e^{-t}h_t(\mathbf{x})t^{s\slash 2}\,\frac{dt}{t} \text{ and } J^{\{s\}}(\mathbf{x},\mathbf{y})={\Gamma \Big(\frac{s}{2}\Big)^{-1}}\int_0^{\infty}e^{-t}h_t(\mathbf{x},\mathbf{y})t^{s\slash 2}\,\frac{dt}{t}.
\end{equation}
Since $\xi \longmapsto (1+\|\xi\|^{2})^{-s/2}$ is radial, thanks to~\eqref{eq:T_xi}, for all $1 \leq j \leq N$ we have
\begin{equation}\label{eq:T_jJ}
    |T_j(1+\|\xi\|^{2})^{-s/2}|=|\partial_j(1+\|\xi\|^{2})^{-s/2}| 
    \leq C(1+\|\xi\|^{2})^{-(s+1)/2}
 \leq C(1+\|\xi\|^{2})^{-s/2}.
\end{equation}

\section{Some formulas and estimates for Dunkl translations of regular enough functions}\label{sec:Estimates}

{In the present section we prove formulas and derive basic estimates for translations of certain functions. Then, in the next section, we shall use them for more advanced estimations.

We start by the following lemma, which  is a consequence of the generalized heat kernel regularity estimates~\eqref{eq:h_t_regular}.

\begin{lemma}\label{lem:E_square}
Let $\varepsilon_1 \in (0,1]$. There is a constant $C>0$ such that for all $t>0$ and $\mathbf{y},\mathbf{y}' \in \mathbb{R}^N$, we have
\begin{equation}\label{eq:E_square}
    \left(\int_{B(0,1/t)}|E(-i\xi,\mathbf{y})|^2\,dw(\xi)\right)^{1/2} \leq
    \frac{C}{w(B(\mathbf{y},t))^{1/2}}
    ,
\end{equation} 
\begin{equation}\label{eq:E_square_lip}
    \left(\int_{B(0,1/t)}|E(-i\xi,\mathbf{y})-E(-i\xi,\mathbf{y}')|^2\,dw(\xi)\right)^{1/2} \leq \left(\frac{\|\mathbf{y}-\mathbf{y}'\|}{t}\right)^{\varepsilon_1}
    \left(\frac{C}{w(B(\mathbf{y},t))^{1/2}}+ \frac{C}{w(B(\mathbf{y}',t))^{1/2}}\right)
    .
\end{equation} 

\end{lemma}

\begin{proof}
We prove just~\eqref{eq:E_square_lip}, the proof of~\eqref{eq:E_square} is analogous (in fact, it was proved in~\cite[(3.6)]{DzH}). By~\eqref{eq:ht_by_Fourier}, the Plancherel's equality~\eqref{eq:Plancherel}, and~\eqref{eq:h_t_regular} we get
\begin{align*}
     & \left(\int_{B(0,1/t)}|E(-i\xi,\mathbf{y})-E(-i\xi,\mathbf{y}')|^2\,dw(\xi)\right)^{1/2} \\
     &\leq  e\left(\int_{B(0,1/t)}|E(-i\xi,\mathbf{y})-E(-i\xi,\mathbf{y}')|^2e^{-2t^2\|\xi\|^2}\,dw(\xi)\right)^{1/2}\\& \leq 
     e\left(\int_{\mathbb{R}^N}|E(-i\xi,\mathbf{y})-E(-i\xi,\mathbf{y}')|^2e^{-2t^2\|\xi\|^2}\,dw(\xi)\right)^{1/2}=e\left(\int_{\mathbb{R}^N}|h_{t^2}(\mathbf{x},\mathbf{y})-h_{t^2}(\mathbf{x},\mathbf{y}')|^2\,dw(\mathbf{x})\right)^{1/2}\\& \leq  C\left(\frac{\|\mathbf{y}-\mathbf{y}'\|}{t}\right)^{\varepsilon_1}\left(\int_{\mathbb{R}^N}|h_{2t^2}(\mathbf{x},\mathbf{y})|^2\,dw(\mathbf{x})\right)^{1/2} + C\left(\frac{\|\mathbf{y}-\mathbf{y}'\|}{t}\right)^{\varepsilon_1}\left(\int_{\mathbb{R}^N}|h_{2t^2}(\mathbf{x},\mathbf{y}')|^2\,dw(\mathbf{x})\right)^{1/2}\\
     & \leq C'\Big( \frac{1}{w(B(\mathbf{y},t))^{1/2}}
     + \frac{1}{w(B(\mathbf{y}',t))^{1/2}}\Big)\left(\frac{\|\mathbf{y}-\mathbf{y}'\|}{t}\right)^{\varepsilon_1}.
\end{align*}
\end{proof}

In order to estimate translations of non-radial functions we need  further preparation. 
The following  lemma and its  proof, which  is based on the fundamental theorem of calculus (see e.g.~\cite[pages 284-285]{DzH-square}), will play a crucial role in our study.
\begin{lemma}\label{lem:diff}
Let $\ell \in \mathbb{N}_0$, $M>0$. If $f \in C^{\ell+1}(\mathbb{R}^N)$ is such that $\partial_jf$ are bounded functions  for all $1 \leq j \leq N$, then 
\begin{align*}
    f^{\{\alpha\}}(\mathbf{x}):=\frac{f(\mathbf{x})-f(\sigma_{\alpha}(\mathbf{x}))}{\langle \alpha,\mathbf{x} \rangle}
\end{align*} 
belongs to $C^{\ell}(\mathbb{R}^N)$ for all $\alpha \in R$ and there is a constant $C>0$ independent of $f$ such that
\begin{align*}
    \|f^{\{\alpha\}}\|_{L^{\infty}} \leq C\sum_{j=1}^N\|\partial_j f\|_{L^{\infty}}.
\end{align*}
Moreover, there is a constant $C>0$ independent of $\ell$ and $f$ such that if 
$$ |\partial^\beta f(\mathbf{x})| \leq (1+\|\mathbf{x}\|)^{-\mathbf{N}-M}\quad \text{ for all \ }  |\beta|\leq \ell +1 $$  then $|T^\beta f^{\{\alpha\}}(\mathbf{x})| \leq C(1+\|\mathbf{x}\|)^{-\mathbf{N}-M}$ for all $|\beta |\leq \ell $,  $\alpha \in R$, and $\mathbf{x} \in \mathbb{R}^N$. 
\end{lemma}

\begin{proposition}\label{propo:formula}
Let $\phi \in \mathcal{S}(\mathbb{R}^N)$ and $1 \leq j \leq N$. Then for all $\mathbf{x},\mathbf{y} \in \mathbb{R}^N$ we have
\begin{equation}\label{eq:key_formula}
    i(x_j-y_j)\phi(\mathbf{x},\mathbf{y})=-\phi_j(\mathbf{x},\mathbf{y})-\sum_{\alpha \in R}\frac{k(\alpha)}{2}\langle \alpha, e_j\rangle \phi_{\alpha}(\mathbf{x},\sigma_{\alpha}(\mathbf{y})),
\end{equation}
where $\phi_j$, $\phi_\alpha$ are Schwartz class functions defined by 
\begin{equation}\label{eq:key_formula_functions}
    \mathcal{F}\phi_j(\xi)=\partial_{j,\xi}\mathcal{F}\phi(\xi), \ \ \mathcal{F}\phi_{\alpha}(\xi)=\frac{\mathcal{F}\phi(\xi)-\mathcal{F}\phi(\sigma_{\alpha}(\xi))}{\langle \xi,\alpha \rangle}.
\end{equation}
Moreover, if $\phi$ is $G$-invariant, then
\begin{equation}\label{eq:key_formula_G_invariant}
    i(x_j-y_j)\phi(\mathbf{x},\mathbf{y})=-\phi_j(\mathbf{x},\mathbf{y}),
\end{equation}
where $\mathcal{F}\phi_j(\xi)=\partial_{j,\xi}\mathcal{F}\phi(\xi)=T_{j,\xi}\mathcal{F}\phi(\xi)$, i.e. $\phi_j(\mathbf{x})=-ix_j\phi(\mathbf{x})$.
\end{proposition}

\begin{proof}
It is obvious, that $\phi_j$ defined in \eqref{eq:key_formula_functions} belong to $\mathcal S(\mathbb R^N)$.
  Further, the functions 
  \begin{align*}
      \mathbb R^N\ni \xi\mapsto \frac{\mathcal{F}\phi(\xi)-\mathcal{F}\phi(\sigma_{\alpha}(\xi))}{\langle \xi,\alpha \rangle}
  \end{align*}
belong the Schwartz class (see Lemma~\ref{lem:diff}). Hence, $\phi_\alpha\in\mathcal S(\mathbb R^N)$ for all $\alpha\in R$. Thanks to the inverse formula and definition of Dunkl kernel (see~\eqref{eq:inverse_teo} and~\eqref{eq:E_def}) we get
\begin{align*}
    ix_{j}\phi(\mathbf{x},\mathbf{y})&=\mathbf{c}_k^{-1}\int_{\mathbb{R}^N}ix_{j}E(i\xi,\mathbf{x})E(i\xi,-\mathbf{y})\mathcal{F}\phi(\xi)\,dw(\xi)\\&=\mathbf{c}_{k}^{-1}\int_{\mathbb{R}^N}\big(T_{j,\xi}[E(i\xi,\mathbf{x})]\big)E(i\xi,-\mathbf{y})\mathcal{F}\phi(\xi)\,dw(\xi).
\end{align*}
It follows from~\eqref{eq:E_est} that for fixed $\mathbf{x} \in \mathbb{R}^N$ we have $(E(-i\cdot,\mathbf{x})\mathcal{F}\phi(\cdot)) \in \mathcal{S}(\mathbb{R}^N)$. Hence, by the integration by parts formula~\eqref{eq:by_parts} and the Leibniz-type rule~\eqref{eq:general_Leibniz} we get
\begin{equation}\label{eq:integration_by_parts}
\begin{split}
    ix_{j}\phi(\mathbf{x},\mathbf{y})&=-\mathbf{c}_{k}^{-1}\int_{\mathbb{R}^N}E(i\xi,\mathbf{x})T_{j,\xi}[E(i\xi,-\mathbf{y})(\mathcal{F}\phi)(\xi)]\,dw(\xi)\\&=-\mathbf{c}_{k}^{-1}\int_{\mathbb{R}^N}E(i\xi,\mathbf{x})
      T_{j,\xi}E(i\xi,-\mathbf{y})
    \mathcal{F}\phi(\xi)\,dw(\xi)\\&\quad-\mathbf{c}_{k}^{-1}\int_{\mathbb{R}^N}E(i\xi,\mathbf{x})
    E(i\xi,-\mathbf{y})\partial_{j,\xi}(\mathcal{F}\phi)(\xi)\,dw(\xi)\\&\quad-\mathbf{c}_{k}^{-1}\int_{\mathbb{R}^N}E(i\xi,\mathbf{x})\sum_{\alpha \in R}\frac{k(\alpha)}{2}\langle \alpha, e_j\rangle E(i\xi,-\sigma_{\alpha}(\mathbf{y}))\frac{(\mathcal{F}\phi)(\xi)-(\mathcal{F}\phi)(\sigma_{\alpha}(\xi))}{\langle \xi,\alpha \rangle }\,dw(\xi).
\end{split}
\end{equation}
Using~\eqref{eq:E_def} and inverse formula~\eqref{eq:inverse_teo} we obtain
\begin{equation}\label{eq:system_again}
\begin{split}
    &-\mathbf{c}_{k}^{-1}\int_{\mathbb{R}^N}E(i\xi,\mathbf{x})(\mathcal{F}\phi)(\xi)T_{j,\xi}E(i\xi,-\mathbf{y})\,dw(\xi)\\&=-\mathbf{c}_{k}^{-1}\int_{\mathbb{R}^N}E(i\xi,\mathbf{x})(\mathcal{F}\phi)(\xi)[-iy_jE(i\xi,-\mathbf{y})]\,dw(\xi)=iy_j\phi(\mathbf{x},\mathbf{y}).
\end{split}
\end{equation}
Therefore,~\eqref{eq:key_formula} is a consequence of~\eqref{eq:integration_by_parts} and~\eqref{eq:system_again}. The proof  of~\eqref{eq:key_formula_G_invariant} follows from~\eqref{eq:key_formula} and~\eqref{eq:key_formula_functions}, since $\mathcal F\phi$ is $G$-invariant, so $\phi_\alpha  \equiv 0$ and $\partial_{j,\xi}\mathcal F\phi (\xi)=T_j\mathcal F\phi (\xi)$ in this case. 
\end{proof}

Let us note that Proposition~\ref{propo:formula} together with its proof can be generalized to  $\phi$ which not necessary belongs to $\mathcal{S}(\mathbb{R}^N)$, but the quantities which appears in the proof make sense. One of such a possible generalization is presented in the proposition below, which will be used in the proof of Theorem~\ref{teo:is_CZ}.

\begin{proposition}
Let $\delta>0$. Assume that $f\in L^1(dw)$ is compactly supported and $g\in L^1(dw)$ is $G$-invariant function such that $|\mathcal Fg(\xi)|\leq (1+\|\xi\|)^{-\mathbf N-\delta}$,  {$\mathcal Fg\in C^1(\mathbb R^N)$}, and $|T_{j}\mathcal Fg(\xi)|\leq (1+\|\xi\|)^{-\mathbf N-\delta}$ for all $1\leq j \leq N$and $\xi \in \mathbb{R}^N$. Then 
\begin{equation}\begin{split}\label{eq:decomposition} 
    &i(x_j-y_j)(f*g)(\mathbf{x},\mathbf{y})=- \mathbf{c}_k^{-1}
    \int_{\mathbb{R}^N} E(i\xi, \mathbf x)E(-i\xi,\mathbf y) (\partial_j \mathcal Ff)(\xi)\mathcal Fg(\xi)\, dw(\xi) \\
    &-\mathbf{c}_k^{-1}\sum_{\alpha \in R} \frac{k(\alpha)}{2}\langle \alpha, e_j\rangle \int_{\mathbb{R}^N} E(i\xi, \mathbf x) \frac{(\mathcal{F}f)(\xi)-(\mathcal{F}f)(\sigma_{\alpha}(\xi))}{\langle \xi,\alpha \rangle} E(-i\xi, \sigma_{\alpha}(\mathbf y)) \mathcal Fg(\xi)\, dw(\xi)\\
    &-\mathbf{c}_k^{-1} \int_{\mathbb{R}^N} E(i\xi,\mathbf x) E(-i\xi, \mathbf y) \mathcal Ff(\xi) (T_{j} \mathcal Fg)(\xi)\, dw(\xi).
\end{split} 
\end{equation}
\end{proposition}

\begin{proof}
First, let us observe that for every multi index $\nu \in \mathbb{N}_0^N$, a function $f\in L^1(dw)$, $\text{supp}\, f\subseteq B(0,r)$,  and $\xi \in \mathbb{R}^N$ one has 
\begin{equation}\label{eq:mathcalF_der}
   |\partial^\nu \mathcal Ff(\xi) |\leq \mathbf{c}_k^{-1}r^{|\nu|} \| f\|_{L^1(dw)}.  
\end{equation}
Indeed, by~\eqref{eq:E_est},
\begin{equation}\label{eq:partial_F}
\begin{split}
    |\partial^\nu \mathcal Ff(\xi) |&=\left|\mathbf{c}_k^{-1}\partial^{\nu}\int_{\mathbb{R}^N}E(-i\xi,\mathbf{x})f(\mathbf{x})\,dw(\mathbf{x})\right|=\left|\mathbf{c}_k^{-1}\int_{B(0,r)}\partial^{\nu}_{\xi}E(-i\xi,\mathbf{x})f(\mathbf{x})\,dw(\mathbf{x})\right| \\&\leq \mathbf{c}_k^{-1} \int_{B(0,r)}\|\mathbf{x}\|^{|\nu|}|f(\mathbf{x})|\,dw(\mathbf{x}) \leq \mathbf{c}_k^{-1} r^{|\nu|}\|f\|_{L^1(dw)}.
\end{split}
\end{equation}
 Similarly, by Lemma~\ref{lem:diff},  
\begin{equation}\label{eq:Ffalpha}
\Big|\frac{(\mathcal{F}f)(\xi)-(\mathcal{F}f)(\sigma_{\alpha}(\xi))}{\langle \xi,\alpha \rangle} \Big|\leq C\sum_{j=1}^N \| \partial_j \mathcal Ff\|_{L^\infty} \leq Cr\|f\|_{L^1(dw)}     
\end{equation}  
Consequently, all of the integrals in~\eqref{eq:decomposition} can be interpreted as the Dunkl transforms of $L^1(dw)$-functions. Hence, in order to establish~\eqref{eq:decomposition}, it is enough to note that applying the Leibniz-type rule~\eqref{eq:general_Leibniz} twice:  firstly  to the functions:  $E(-i \cdot,\mathbf{y})\mathcal{F}f$ (not necessarily $G$-invariant)  and $\mathcal{F}g$ ($G$-invariant) and then to the functions $E(-i\cdot,\mathbf{y})$ and $\mathcal{F}f$,  we obtain
\begin{align*}
    T_{j,\xi}(E(-i\cdot,\mathbf{y})(\mathcal{F}f)(\mathcal{F}g))(\xi)&=T_{j,\xi}(E(-i\xi,\mathbf{y}))(\xi)(\mathcal{F}f)(\xi)(\mathcal{F}g)(\xi)\\&+E(-i\xi,\mathbf{y})\partial_{j,\xi}(\mathcal{F}f)(\xi)(\mathcal{F}g)(\xi)\\&+\sum_{\alpha \in R}\frac{k(\alpha)}{2}\langle \alpha, e_j\rangle \frac{(\mathcal{F}f)(\xi)-(\mathcal{F}f)(\sigma_{\alpha}(\xi))}{\langle \xi,\alpha\rangle}E(-i\xi,\sigma_{\alpha}(\mathbf{y}))(\mathcal{F}g)(\xi)\\&+E(-i\xi,\mathbf{y})(\mathcal{F}f)(\xi)T_{j,\xi}(\mathcal{F}g)(\xi),
\end{align*}
and repeat the proof of Proposition~\ref{propo:formula}.
\end{proof}

\begin{proposition}\label{propo:main_xy}
Let $\delta>0$ and $0<\varepsilon_1 \leq 1$. Assume that $f \in L^1(dw)$ and  $g\in L^1(dw)$ is such that $|\mathcal Fg(\xi)|\leq (1+\|\xi\|)^{-\mathbf N-\delta}$ for all $\xi \in \mathbb{R}^N$. Then the following statements hold.
\begin{enumerate}[(a)]
\item{There is a constant $C_1>0$ independent of $f,g$ such that for all $1 \leq j \leq N$ and $\mathbf{x},\mathbf{y} \in \mathbb{R}^N$, one has
    \begin{equation}\label{eq:without_xy}
   |(f*g)(\mathbf x,\mathbf y)|\leq C w(B(\mathbf x,1))^{-1/2} w(B(\mathbf y, 1))^{-1/2}\| f\|_{L^1(dw)}. 
\end{equation}}
    \item{If additionally $g$ is $G$-invariant, $ \mathcal Fg\in C^1(\mathbb R^N)$, and satisfies $|T_{j}\mathcal Fg(\xi)|\leq (1+\|\xi\|)^{-\mathbf N-\delta}$ for all $\xi \in \mathbb{R}^N$, then there is a constant $C_2>0$ independent of $f,g$ such that for all $f\in L^1(dw)$ such that  $\text{supp}\, f\subseteq B(0,r)$ and $\mathbf{x},\mathbf{y} \in \mathbb{R}^N$, we have 
    \begin{equation}\label{eq:with_xy}
    |(x_j-y_j) (f*g)(\mathbf x,\mathbf y)|\leq C_2r w(B(\mathbf x,1))^{-1/2} w(B(\mathbf y, 1))^{-1/2}\| f\|_{L^1(dw)}.
\end{equation}
}   
    \item{Assume $\delta>\varepsilon_1$. If $g$ is $G$-invariant, $ \mathcal Fg\in C^1(\mathbb R^N)$, and  $|T_{j}\mathcal Fg(\xi)|\leq (1+\|\xi\|)^{-\mathbf N-\delta}$ for all $\xi \in \mathbb{R}^N$, then there is a constant $C_3>0$ independent of $f,g$ such that for all $f\in L^1(dw)$ such that  $\text{supp}\, f\subseteq B(0,r)$ and  $\mathbf{x},\mathbf{y},\mathbf{y}' \in \mathbb{R}^N$, we have
     \begin{equation}\label{eq:with_xy_lip}
     \begin{split}
    |x_j-y_j|| (f*g)(\mathbf x,\mathbf y)-(f*g)(\mathbf x,\mathbf y')|&\leq C_3r\|\mathbf{y}-\mathbf{y}'\|^{\varepsilon_1} w(B(\mathbf x,1))^{-1/2} w(B(\mathbf y, 1))^{-1/2}\| f\|_{L^1(dw)}\\&+  C_3r\|\mathbf{y}-\mathbf{y}'\|^{\varepsilon_1} w(B(\mathbf x,1))^{-1/2} w(B(\mathbf y', 1))^{-1/2}\| f\|_{L^1(dw)}.
    \end{split}
\end{equation}
    
    }
\end{enumerate} 
\end{proposition} 

\begin{proof}
Let $U_0=B(0,1)$ and $U_\ell=B(0,2^\ell)\setminus B(0,2^{\ell-1})$ for $\ell \in \mathbb{N}$. In order to prove~\eqref{eq:without_xy}, we use the Cauchy--Schwarz inequality,~\eqref{eq:E_square}, and~\eqref{eq:growth} (cf.~\cite[Proposition 3.7]{DzH}),
\begin{equation}\label{eq:37}
\begin{split}
    &|f*g(\mathbf{x},\mathbf{y})|=\left|\mathbf{c}_k^{-1}\int_{\mathbb{R}^N} E(i\xi, \mathbf x)  E(-i\xi,\mathbf y) (\mathcal Ff)(\xi)\mathcal Fg(\xi)\, dw(\xi) \right|\\ &\leq \sum_{\ell=0}^\infty  \mathbf{c}_k^{-1}\Big|  \int_{U_\ell} E(i\xi, \mathbf x)E(-i\xi,\mathbf y) (\mathcal Ff)(\xi)\mathcal Fg(\xi)\,dw(\xi) \Big| \\&\leq \sum_{\ell=0}^{\infty}\mathbf{c}_k^{-1}\|\mathcal{F}f\|_{L^{\infty}}\left(\int_{U_\ell} \frac{|E(i\xi,\mathbf{x})|^2}{(1+\|\xi\|)^{2\mathbf{N}+2\delta}}\,dw(\xi)\right)^{1/2}\left(\int_{B(0,2^\ell)}|E(-i\xi,\mathbf{y})|^2\,dw(\xi)\right)^{1/2}
    \\
       &\leq C\sum_{\ell=0}^\infty 2^{-\ell(\mathbf N+\delta)} w(B(\mathbf x,2^{-\ell}))^{-1/2}w(B(\mathbf y,2^{-\ell}))^{-1/2}\|f\|_{L^1(dw)}\\&\leq C'w(B(\mathbf x,1))^{-1/2} w(B(\mathbf y, 1))^{-1/2}\|f\|_{L^1(dw)},
\end{split}
\end{equation}
so~\eqref{eq:without_xy} is proved. In order to prove~\eqref{eq:with_xy}, we use~\eqref{eq:decomposition}. We shall estimate the first component of the right-hand side of~\eqref{eq:decomposition}, the others are treated in the same way.  Recall that $\|\partial_j \mathcal{F}f\|_{L^{\infty}} \leq { \mathbf{c}_k^{-1}}r\|f\|_{L^1(dw)}$ (see~\eqref{eq:partial_F}). Therefore, similarly as in~\eqref{eq:37}, we obtain
\begin{equation}\label{eq:37_1}
    \begin{split}
       \Big|  \int_{\mathbb{R}^N} E(i\xi, \mathbf x) & E(-i\xi,\mathbf y) (\partial_j \mathcal Ff)(\xi)\mathcal Fg(\xi)\, dw(\xi) \Big| \\
       &\leq \sum_{\ell=0}^\infty  \Big|  \int_{U_\ell} E(i\xi, \mathbf x)E(-i\xi,\mathbf y) (\partial_j \mathcal Ff)(\xi)\mathcal Fg(\xi)\, dw(\xi) \Big| \\
       &\leq Cr\sum_{\ell=0}^\infty 2^{-\ell(\mathbf N+\delta)} rw(B(\mathbf x,2^{-\ell}))^{-1/2}w(B(\mathbf y,2^{-\ell}))^{-1/2}\|f\|_{L^1(dw)}\\&\leq C'rw(B(\mathbf x,1))^{-1/2} w(B(\mathbf y, 1))^{-1/2}\|f\|_{L^1(dw)}.
    \end{split}
\end{equation}
We now turn  to prove~\eqref{eq:with_xy_lip}. We write
\begin{equation*}
\begin{split}
     |x_j-y_j|| (f*g)(\mathbf x,\mathbf y)-(f*g)(\mathbf x,\mathbf y')|& \leq |(x_j-y_j) (f*g)(\mathbf x,\mathbf y)-(x_j-y_j')(f*g)(\mathbf x,\mathbf y')|\\&+ |y'_j-y_j||(f*g)(\mathbf x,\mathbf y')|=:I_1+I_2.
    \end{split}
\end{equation*}
The required estimate for $I_2$ follows from~\eqref{eq:without_xy}. To deal with  $I_1$, we use~\eqref{eq:decomposition} and  obtain
\begin{equation}
    \begin{split}
        &I_1 \leq \mathbf{c}_k^{-1}
    \int_{\mathbb{R}^N} |E(i\xi, \mathbf x)||E(-i\xi,\mathbf y)-E(-i\xi,\mathbf{y}')| |(\partial_{j, \xi} \mathcal Ff)(\xi)||\mathcal Fg(\xi)|\, dw(\xi) \\
    &+ \mathbf{c}_k^{-1}\sum_{\alpha \in R} \frac{k(\alpha)}{2}|\langle \alpha, e_j\rangle |\int_{\mathbb{R}^N} |E(i\xi, \mathbf x)| \left|\frac{(\mathcal{F}f)(\xi)-(\mathcal{F}f)(\sigma_{\alpha}(\xi))}{\langle \xi,\alpha \rangle}\right|\\
    &\hskip 5cm \times | E(-i\xi, \sigma_{\alpha}(\mathbf y))-E(-i\xi, \sigma_{\alpha}(\mathbf y'))| |\mathcal Fg(\xi)|\, dw(\xi)\\
    &+\mathbf{c}_k^{-1} \int_{\mathbb{R}^N} |E(i\xi,\mathbf x)||E(-i\xi, \mathbf y)-E(-i\xi,\mathbf{y}')| |\mathcal Ff(\xi)|| (T_{j} \mathcal Fg)(\xi)|\, dw(\xi)=:I_{1,1}+I_{1,2}+I_{1,3}.
    \end{split}
\end{equation}
In order to estimate $I_{1,1}$, we proceed similarly to~\eqref{eq:37} and~\eqref{eq:37_1}. By the Cauchy--Schwarz inequality together with~\eqref{eq:E_square},~\eqref{eq:E_square_lip}, and~\eqref{eq:growth} we have
\begin{align*}
    &I_{1,1} \leq \mathbf{c}_k^{-1}\sum_{\ell=0}^\infty    \int_{U_\ell} |E(i\xi, \mathbf x)| |(E(-i\xi,\mathbf y)-E(-i\xi,\mathbf y'))| |(\partial_{j,\xi}\mathcal Ff)(\xi)||\mathcal Fg(\xi)|\,dw(\xi)  \\&\leq \sum_{\ell=0}^{\infty} \mathbf{c}_k^{-1}\|\partial_{j,\xi}\mathcal{F}f\|_{L^{\infty}}\left(\int\limits_{U_\ell}\frac{|E(i\xi,\mathbf{x})|^2}{(1+\|\xi\|)^{2\mathbf{N}+2\delta}}\,dw(\xi)\right)^{1/2}\left(\int\limits_{B(0,2^\ell)}|E(-i\xi,\mathbf{y})-E(-i\xi,\mathbf{y}')|^2\,dw(\xi)\right)^{1/2} 
    \\
       &\leq Cr\|\mathbf{y}-\mathbf{y}'\|^{\varepsilon_1}\sum_{\ell=0}^\infty 2^{-\ell(\mathbf N+\delta-\varepsilon_1)} w(B(\mathbf x,2^{-\ell}))^{-1/2}(w(B(\mathbf y,2^{-\ell}))^{-1/2}+w(B(\mathbf y',2^{-\ell}))^{-1/2})\|f\|_{L^1(dw)}\\&\leq C'r\|\mathbf{y}-\mathbf{y}'\|^{\varepsilon_1}w(B(\mathbf x,1))^{-1/2} (w(B(\mathbf y, 1))^{-1/2}+w(B(\mathbf y', 1))^{-1/2})\|f\|_{L^1(dw)}.
\end{align*}
The estimate for $I_{1,3}$ goes identically. In order to deal with $I_{1,2}$, we recall that 
\begin{align*}
\left|\frac{(\mathcal{F}f)(\xi)-(\mathcal{F}f)(\sigma_{\alpha}(\xi))}{\langle \xi,\alpha \rangle}\right|
     \leq Cr \| f\|_{L^1(dw)} \text{ for all } \xi \in \mathbb{R}^N
\end{align*} 
(see~\eqref{eq:Ffalpha}).
Moreover, $\|\sigma_{\alpha}(\mathbf{y})-\sigma_{\alpha}(\mathbf{y}')\|=\|\mathbf{y}-\mathbf{y}'\|$  for all $\mathbf{y},\mathbf{y}' \in \mathbb{R}^N$ and $\alpha \in R$. Consequently, for $I_{1,2}$ one can repeat the same proof as for $I_{1,1}$.
\end{proof}
}

Since any sufficiently regular  function can be written as a convolution of a nice radial function  with an $L^1$-function, as a consequence of Proposition~\ref{propo:main_xy} we obtain the following theorem. 
\begin{theorem}\label{teo:main_xy_single_f}
Let $s$ be an even integer greater than $\mathbf{N}$. Then for any $0 \leq \varepsilon_1 < s-\mathbf{N}$, $\varepsilon_1 \leq 1$, there is a constant $C>0$ such that for all $f \in C^{s}(\mathbb{R}^N)$ such that $\supp f \subseteq B(0,1)$, and for all $\mathbf{x},\mathbf{y},\mathbf{y}' \in \mathbb{R}^N$ we have
\begin{equation}\label{eq:f_Cs}
    |f(\mathbf{x},\mathbf{y})| \leq C\|f\|_{C^s(\mathbb{R}^N)}(1+\|\mathbf{x}-\mathbf{y}\|)^{-1}w(B(\mathbf{x},1))^{-1/2}w(B(\mathbf{y},1))^{-1/2}\chi_{[0,1]}(d(\mathbf x,\mathbf y)),
\end{equation}
\begin{equation}\label{eq:f_Cs_lip}
\begin{split}
    |f(\mathbf{x},\mathbf{y})-f(\mathbf{x},\mathbf{y}')| \leq C\frac{\|f\|_{C^s(\mathbb{R}^N)}\|\mathbf{y}-\mathbf{y}'\|^{\varepsilon_1}}{(1+\|\mathbf{x}-\mathbf{y}\|)^{\varepsilon_1}}w(B(\mathbf{x},1))^{-1/2}\left(w(B(\mathbf{y},1))^{-1/2}+w(B(\mathbf{y}',1))^{-1/2}\right).
\end{split}
\end{equation}
\end{theorem}

\begin{proof}
 For $\mathbf{x},\mathbf{y} \in \mathbb{R}^N$ we write
\begin{align*}
    f(\mathbf{x},\mathbf{y})&=\mathbf{c}_k^{-1}\int_{\mathbb{R}^N}E(i\xi,\mathbf{x})E(-i\xi,\mathbf{y})(\mathcal{F}f)(\xi)\,dw(\xi)\\&=\mathbf{c}_k^{-1}\int_{\mathbb{R}^N}E(i\xi,\mathbf{x})E(-i\xi,\mathbf{y})\left[(\mathcal{F}f)(\xi)(1+\|\xi\|^2)^{s/2}\right](1+\|\xi\|^2)^{-s/2}\,dw(\xi)\\&=\mathbf{c}_k \left(\widetilde{f}*J^{\{s\}}\right)(\mathbf{x},\mathbf{y}),
\end{align*}
where $J^{\{s\}}$ is defined in~\eqref{eq:bessel} and
\begin{align*}
    \mathcal{F}\widetilde{f}(\xi)=(\mathcal{F}f)(\xi)(1+\|\xi\|^2)^{s/2}.
\end{align*}
Therefore, by~\eqref{eq:T_j_fourier_side} we have $\widetilde{f}=(1-\Delta_k)^{s/2}f$. Consequently, by the assumption $\supp f \subseteq B(0,1)$ and Lemma~\ref{lem:diff}, there is a constant $C>0$ such that 
\begin{equation}\label{eq:tilde_key}
    \|\widetilde{f}\|_{L^1} \leq C\|f\|_{C^{s}(\mathbb{R}^N)}.
\end{equation}
Hence, applying Proposition~\ref{propo:main_xy} with $\widetilde{f}$, $J^{\{s\}}$ (which is $G$-invariant),  $\delta:=s-\mathbf{N}$, and any  $0<\varepsilon_1 <\delta$ (the assumptions are satisfied thanks to the definition of $J^{\{s\}}$ and~\eqref{eq:T_jJ}), we obtain~\eqref{eq:f_Cs} and~\eqref{eq:f_Cs_lip}.
\end{proof}

\section{Applications of formulas and estimates from Section\texorpdfstring{~\ref{sec:Estimates}}{3}}

\subsection{Estimates for Dunkl translations of Schwartz--class functions}
{As a consequence of Theorem~\ref{teo:main_xy_single_f}, we obtain the following theorem.}
\begin{theorem}\label{teo:trans-g}
   Let $s$ be an even integer greater than   $\mathbf{N}$. Assume that for a  certain $ \kappa \geq  {-\mathbf N/2-1}$  and a function $g \in C^s(\mathbb{R}^N)$ one has 
   \begin{equation}\label{eq:g_t_assumtions}
       |\partial^{\beta}g(\mathbf{x})| \leq (1+\|\mathbf{x}\|)^{-\mathbf{N}-|\beta|-1-\kappa} \text{ for all }\mathbf{x} \in \mathbb{R} \text{ and }|\beta| \leq s.
   \end{equation}
   Then there is a constant $C>0$ {(independent of $g$)} such that for all $\mathbf{x},\mathbf{y} \in \mathbb{R}^N$ and $t>0$ we have 
   \begin{equation}\label{eq:g_t}
       |g_t(\mathbf{x},\mathbf{y})| \leq C\left(1+\frac{\|\mathbf{x}-\mathbf{y}\|}{t}\right)^{-1}\Big(1+\frac{d(\mathbf x,\mathbf y)}{t}\Big)^{-\kappa} \frac{1}{w(B(\mathbf{x},d(\mathbf{x},\mathbf{y})+t))}, 
   \end{equation}
   where $  g_t(\mathbf x)=t^{-\mathbf N}g(\mathbf x/t)$.   
\end{theorem}
\begin{remark}
    Let us note that 
{by \eqref{eq:growth},
  $   w(B(\mathbf x, t+d(\mathbf x,\mathbf y)))^{-1} \leq  w(B(\mathbf x, {t}))^{-1} (1+d(\mathbf x,\mathbf y)/t)^{ -N}$ hence, under assumptions of Theorem \ref{teo:trans-g}, we have  
   \begin{equation}\label{eq:g_t2}
       |g_t(\mathbf{x},\mathbf{y})| \leq C\left(1+\frac{\|\mathbf{x}-\mathbf{y}\|}{t}\right)^{-1}\Big(1+\frac{d(\mathbf x,\mathbf y)}{t}\Big)^{-N-\kappa} \frac{1}{w(B(\mathbf{x},t))}, 
   \end{equation} }
   \end{remark}
\begin{proof}[Proof of Theorem \ref{teo:trans-g}]
By scaling it is enough to prove~\eqref{eq:g_t} for $t=1$. Let $\widetilde{\Psi}_0 \in C_c^{\infty}((-\frac{1}{2},\frac{1}{2}))$ and $\widetilde{\Psi} \in C_c^{\infty}((\frac{1}{8},1))$ be such that 
\begin{equation}\label{eq:reproduce}
1=\widetilde{\Psi}_{0}(\|\mathbf{x}\|)+\sum_{\ell=1}^{\infty}\widetilde{\Psi}(2^{-\ell}\|\mathbf{x}\|)=\sum_{\ell=0}^{\infty}\widetilde{\Psi}_{\ell}(\|\mathbf{x}\|)=:\sum_{\ell=0}^{\infty}\Psi_{\ell}(\mathbf{x}) \text{ for all }\mathbf{x}\in \mathbb{R}^N.
\end{equation}
{Then 
\begin{equation}\label{eq:reproduce_g}
    g(\mathbf x)=\sum_{\ell=0}^\infty g(\mathbf x)\Psi_\ell(\mathbf x)=\sum_{\ell=0}^\infty g_\ell(\mathbf x),
\end{equation} 
where the convergence is in $L^2(dw(\mathbf x))$. By continuity of the generalized translations on $L^2(dw)$ for all $\mathbf y\in\mathbb R^N$  we have 
\begin{equation}\label{eq:reproduce_1} 
g(\mathbf{x},\mathbf{y})=\sum_{\ell=0}^{\infty}(g \cdot \Psi_{\ell})(\mathbf{x},\mathbf{y})=:\sum_{\ell=0}^{\infty}g_\ell(\mathbf{x},\mathbf{y}),
\end{equation}
where the convergence in $L^2(dw(\mathbf x))$. We turn to prove that the series converges  absolutely for all $\mathbf x,\mathbf y\in\mathbb R^N$.} Indeed, for fixed $\ell \in \mathbb{N}_0$ we consider $\widetilde{g}_\ell(\mathbf{x})=g_\ell(2^\ell \mathbf{x})$. Then $\widetilde{g}_\ell$ is supported by $B(0,1)$ and it follows from~\eqref{eq:g_t_assumtions} that there is a constant $C>0$ such that for all $\ell \in \mathbb{N}_0$ we have
\begin{align*}
    \|\partial^{\beta}\widetilde{g}_\ell\|_{L^{\infty}} \leq C2^{-\ell(\mathbf{N}+1+\kappa)}.
\end{align*}
Applying Theorem~\ref{teo:main_xy_single_f} we get
\begin{align*}
    |\widetilde{g}_\ell(\mathbf{x},\mathbf{y})| \leq C2^{-\ell(\mathbf{N}+1+\kappa)}\left(1+\|\mathbf{x}-\mathbf{y}\|\right)^{-1}w(B(\mathbf{x},1))^{-1/2}w(B(\mathbf{y},1))^{-1/2}\chi_{[0,1]}(d(\mathbf{x},\mathbf{y})),
\end{align*}
therefore, by scaling and~\eqref{eq:t_ball},
\begin{align*}
    |g_\ell(\mathbf{x},\mathbf{y})| \leq C2^{-\ell\kappa}\left(2^\ell+\|\mathbf{x}-\mathbf{y}\|\right)^{-1}w(B(\mathbf{x},2^\ell))^{-1/2}w(B(\mathbf{y},2^\ell))^{-1/2}\chi_{[0,2^\ell]}(d(\mathbf{x},\mathbf{y})).
\end{align*}
Finally, by~\eqref{eq:growth},
\begin{align*}
    &\sum_{\ell=0}^{\infty}|g_\ell(\mathbf{x},\mathbf{y})|=\sum_{2^\ell \geq d(\mathbf{x},\mathbf{y}), {\ell\geq 0}}|g_\ell(\mathbf{x},\mathbf{y})| \\
    &\leq C\sum_{2^\ell \geq d(\mathbf{x},\mathbf{y}), {\ell\geq 0}}2^{-\ell\kappa}\left(2^\ell+\|\mathbf{x}-\mathbf{y}\|\right)^{-1}w(B(\mathbf{x},2^\ell))^{-1/2}w(B(\mathbf{y},2^\ell))^{-1/2}\\&\leq C\sum_{2^\ell \geq d(\mathbf{x},\mathbf{y}), { \ell\geq 0}}2^{-\ell\kappa}\frac{(d(\mathbf{x},\mathbf{y})+1)^N}{2^{\ell N}} \left(1+\|\mathbf{x}-\mathbf{y}\|\right)^{-1}\\&\times w(B(\mathbf{x},d(\mathbf{x},\mathbf{y})+1))^{-1/2}w(B(\mathbf{y},d(\mathbf{x},\mathbf{y})+1))^{-1/2} \\&\leq C(1+\|\mathbf{x}-\mathbf{y}\|)^{-1}(1+d(\mathbf{x},\mathbf{y}))^{-\kappa}w(B(\mathbf{x},d(\mathbf{x},\mathbf{y})+1))^{-1},
\end{align*}
where in the last step we have used the fact that the quantities $w(B(\mathbf{x},d(\mathbf{x},\mathbf{y})+1))$ and $w(B(\mathbf{y},d(\mathbf{x},\mathbf{y})+1))$ are comparable.
\end{proof}

Assume $\varphi \in \mathcal{S}(\mathbb{R}^N)$. It follows from Theorem~\ref{teo:trans-g} that for any $M>0$ there is a constant $C_M>0$ such that for all $\mathbf{x},\mathbf{y} \in \mathbb{R}^N$ we have
\begin{equation}\label{eq:without_Lambda}
    |\varphi(\mathbf{x},\mathbf{y})| \leq \frac{C_M}{w(B(\mathbf{x},1))}{\left(1+\|\mathbf{x}-\mathbf{y}\|\right)^{-1}}
    \left(1+d(\mathbf{x},\mathbf{y})\right)^{-M}.
\end{equation}
{Moreover, if additionally a Schwartz class function  $\varphi$ is $G$-invariant, then 
\begin{equation}
    \label{eq:without_Lambda_G}
    |\varphi(\mathbf{x},\mathbf{y})| \leq \frac{C_M}{w(B(\mathbf{x},1))}{\left(1+\|\mathbf{x}-\mathbf{y}\|\right)^{-2}}
    \left(1+d(\mathbf{x},\mathbf{y})\right)^{-M}.
\end{equation}
}
{Let us remark that if $g$ is radial then  the bound for  $\tau_{\mathbf x}(-\mathbf y)$ can be improved under a weaker assumption on $g$. This is stated in the following proposition. 
\begin{proposition}\label{prop:tran_radiall} Assume that $\kappa >2-N$ and $\kappa >-\mathbf N/2$. Then there is a constant $C>0$ such that for all radial functions $g$ satisfying 
     $|g(\mathbf x)|\leq (1+\|\mathbf x\|)^{-\mathbf N-\kappa}$ 
     one has 
     \begin{equation}
|g(\mathbf x,\mathbf y) |\leq C w(B(\mathbf x, 1+d(\mathbf x,\mathbf y)))^{-1} (1+\|\mathbf x-\mathbf y\|)^{-2} (1+d(\mathbf x,\mathbf y))^{-\kappa +2}. 
     \end{equation}
\end{proposition}

\begin{proof} The proof follows the same pattern as that of Theorem~\ref{teo:trans-g}. To this end we note that  
    from the estimates for the Dunkl heat kernel \eqref{eq:intro_heat_2} and the fact that the generalized translation of a non-negative radial function is non-negative combined with Theorem \ref{teo:support} we have 
    \begin{equation}
       | g_\ell(\mathbf x,\mathbf y)|\leq C 2^{-\kappa \ell+2\ell} w(B(\mathbf x,2^\ell))^{-1}  \Big(2^\ell+\| \mathbf x-\mathbf y\|\Big)^{-2}\chi_{[0,2^\ell] }(d(\mathbf x,\mathbf y)),
    \end{equation}
   where $g_\ell$ are define as in \eqref{eq:reproduce_g}. Summing up the estimates we arrive in the desired bound.

\end{proof}

}
{Now} we provide the estimates for the Dunkl translations of the (non-necessarily radial) Schwartz--class functions $\varphi$, which make use of the function $\Lambda(\mathbf{x},\mathbf{y},1)$ (see~\eqref{eq:Lambda_def}). The following lemma was proved in~\cite{DH-heat}.

\begin{lemma}\label{lem:non_A}
For any sequence {$\{\sigma_j\}_{j=0}^{m}$} of elements of the group $G$, {$m \geq |G|^2+1$,} satisfying the condition $\sigma_0={\rm id}$ and
\begin{equation}\label{eq:rekursja}
    \sigma_{j+1}=g_{j+1} \circ \sigma_j \text{ for }j \geq 0,
\end{equation}
where $g_{j+1} \in \{{\rm id}\} \cup \{\sigma_{\alpha}\;:\; \alpha \in R\}$, and $\mathbf{x},\mathbf{y} \in \mathbb{R}^N$, there is a sequence $\boldsymbol{\alpha} \in \mathcal{A}(\mathbf{x},\mathbf{y})$ of elements of $R$ such that $\ell(\boldsymbol{\alpha}) \leq |G|$ and for all $t>0$ we have
\begin{equation}\label{eq:non_A_products}
    {\prod_{j=0}^{m}\left(1+\frac{\|\mathbf{x}-\sigma_j(\mathbf{y})\|}{\sqrt{t}}\right)^{-2} }\leq \rho_{\boldsymbol{\alpha}}(\mathbf{x},\mathbf{y},t) \leq \Lambda (\mathbf x,\mathbf y,t) .
\end{equation}
\end{lemma}

\begin{theorem}\label{teo:Schwartz-transl}
Let $\varphi \in \mathcal{S}(\mathbb{R}^N)$ and $M>0$. Let $\varphi_t:=t^{-\mathbf{N}}\varphi(\cdot/t)$  There is a constant $C_{M,\varphi} >0$ such that for all $\mathbf{x},\mathbf{y} \in \mathbb{R}^N$ and $t>0$, we have
\begin{equation}\label{eq:sqrt_Lambda}
    |\varphi_t(\mathbf{x},\mathbf{y})| \leq C_{M,\varphi} {{\Lambda}(\mathbf{x},\mathbf{y},t^2)^{1/2}}\left(1+\frac{d(\mathbf{x},\mathbf{y})}{t}\right)^{-M}\frac{1}{w(B(\mathbf{x},t))}.
\end{equation}
\end{theorem}

\begin{proof}
By scaling, without loss of generality, we may assume $t=1$. It follows by~\eqref{eq:key_formula} that there is a constant $C>0$ independent of $\mathbf{x},\mathbf{y} \in \mathbb{R}^N$ and $\phi \in \mathcal{S}(\mathbb{R}^N)$ such that
\begin{equation}\label{eq:key_formula_app}
    |\phi(\mathbf{x},\mathbf{y})| \leq C \left(1+\|\mathbf{x}-\mathbf{y}\|\right)^{-1} \left(\sum_{j=1}^{N}|\phi_j(\mathbf{x},\mathbf{y})|+\sum_{\alpha \in {R}}|\phi_{\alpha}(\mathbf{x},\sigma_{\alpha}(\mathbf{y}))|\right),
\end{equation}
where $\phi_j$, $\phi_{\alpha}$ are defined in~\eqref{eq:key_formula_functions}.

Fix a function $\varphi$ from the Schwartz class $\mathcal S(\mathbb R^N)$. In the first step we estimate $\varphi(\mathbf x,\mathbf y)$ by \eqref{eq:key_formula_app}. In the second step we apply the  formula~\eqref{eq:key_formula_app} to $\varphi_j$ and $\varphi_{\alpha}$ obtaining 
\begin{equation*}
    \begin{split}
        |\varphi(\mathbf x,\mathbf y)|&\leq (1+\|\mathbf x-\mathbf y\|)^{-1}\Bigg\{\sum_{j=1}^N(1+\|\mathbf x-\mathbf y\|)^{-1} \Big(\sum_{j_1=1}^N| \varphi_{j,j_1}(\mathbf x,\mathbf y)|+\sum_{\alpha'\in R}|\varphi_{j,\alpha'}(\mathbf x,\sigma_{\alpha}(\mathbf y))|\Big)\\
        &\ \ +\sum_{\alpha\in R}(1+\|\mathbf x-\sigma_\alpha(\mathbf y)\|)^{-1}\Big(\sum_{j_1=1}^N|\varphi_{\alpha,j_1} (\mathbf x,\sigma_\alpha(\mathbf y))|+\sum_{\alpha'\in R}|\varphi_{\alpha,\alpha'}(\mathbf x,\sigma_\alpha'(\sigma_\alpha(\mathbf y)))|\Big)\Bigg\},
    \end{split}
\end{equation*}
where $\varphi_{j,j_1},\ \varphi_{j,\alpha'}, \ \varphi_{\alpha,j_1}, \ \varphi_{\alpha,\alpha'}\in \mathcal S(\mathbb R^N)$. Then we continue this procedure with the use of \eqref{eq:key_formula_app} to estimate $\varphi_{j,j_1},\ \varphi_{j,\alpha'}, \ \varphi_{\alpha,j_1}, \ \varphi_{\alpha,\alpha'}$ and so on.  Set $m=|G|^2$. 
 Let $\mathcal{B}$ be the set of all sequences $\{\sigma_j\}_{j=0}^{m}$ of length $m+1$ satisfying the assumptions of Lemma~\ref{lem:non_A}.  Finally, after all together $(m+1)$--steps {described above}, we get
\begin{equation}\label{eq:non_radial_1}
    |\varphi(\mathbf{x},\mathbf{y})| \leq C' \left(\sum_{\{\sigma_j\}_{j=0}^{m} \in \mathcal{B}}\prod_{j=0}^{m}\left(1+\|\mathbf{x}-\sigma_j(\mathbf{y})\|\right)^{-1}\right)\left(\sum_{\ell=0}^{n}\sum_{g \in G}|\psi_{g,\ell}(\mathbf{x},g(\mathbf{y}))|\right),
\end{equation}
where $\psi_{g,\ell} \in \mathcal{S}(\mathbb{R}^N)$ and $n=(N+|R|)^{m+1}$. Since $d(\mathbf{x},g(\mathbf{y}))=d(\mathbf{x},\mathbf{y})$ (see~\eqref{eq:d}), by~\eqref{eq:without_Lambda} we get
\begin{equation}\label{eq:non_radial_2}
   \left(\sum_{\ell=0}^{n}\sum_{g \in G}|\psi_{g,\ell}(\mathbf{x},g(\mathbf{y}))|\right) \leq C\left(1+d(\mathbf{x},\mathbf{y})\right)^{-M}\frac{1}{w(B(\mathbf{x},1))}.
\end{equation}
Moreover, by Lemma~\ref{lem:non_A} we have
\begin{equation}\label{eq:non_radial_3}
    \sum_{\{\sigma_j\}_{j=0}^{m} \in \mathcal{B}}\prod_{j=0}^{m}\left(1+\|\mathbf{x}-\sigma_j(\mathbf{y})\|\right)^{-1} \leq C\sum_{\boldsymbol{\alpha} \in \mathcal{A}(\mathbf{x},\mathbf{y}), \; \ell(\boldsymbol{\alpha}) \leq |G|}\rho_{\boldsymbol{\alpha}}(\mathbf{x},\mathbf{y},1)^{-1/2} \leq C'\Lambda(\mathbf{x},\mathbf{y},1)^{1/2}.
\end{equation}
Hence, taking into account~\eqref{eq:non_radial_1},~\eqref{eq:non_radial_2}, and~\eqref{eq:non_radial_3} we obtain~\eqref{eq:sqrt_Lambda}.
\end{proof}

\subsection{Singular integral operators}

{Basic examples of singular integral operators are Riesz transforms. The Riesz transforms 
$$ \mathcal R_jf(\mathbf x)=T_j(-\Delta_k)^{-1/2}f(\mathbf x)=\mathcal F^{-1}\left(-i\frac{\xi_j}{\|\xi\|}\mathcal Ff(\xi)\right)(\mathbf x).$$
in the Dunkl setting  were studied by Thangavelu and Xu \cite{ThangaveluXu1} (in dimension 1 and in the product case) and by Amri and Sifi \cite{AS} (in higher dimensions)  who proved the bounds on $L^p(dw)$ spaces. Further, in \cite{ADzH} the Riesz transforms $\mathcal R_j$ were used for characterization of the Hardy space $H^1_{\Delta_k}$. }

Recently, some various approaches to the theory of singular integrals, which use the $d(\mathbf{x},\mathbf{y})$, $\|\mathbf{x}-\mathbf{y}\|$ and $w(B(\mathbf{x},1))$ were investigated. For instance, in~\cite{singular}, the convolution--type singular integrals $f \mapsto K*f$ were studied under some assumptions on the kernel $K$ (see~\eqref{eq:uni_on_annulus},~\eqref{eq:assumption1}, and~\eqref{eq:limitA} in Subsection \ref{sub:conv} below). On the other hand, in~\cite{Tan}, the authors proposed certain assumptions on kernels of non-necessarily convolution--type singular integrals (see \eqref{eq:Tan_1},  \eqref{eq:Tan_2}, \eqref{eq:Tan_3} below) which are relevant for proving some harmonic analysis spirit results in the Dunkl setting. As an example, it was proved there that the kernels  of Riesz transforms $\mathcal R_j$ have the expected  properties. In this section, we will use the results of Section~\ref{sec:Estimates} to unify these two approaches and prove that the kernel estimates of~\cite{Tan} are satisfied for the Dunkl type convolution operators considered  in~\cite{singular}. Consequently, we obtain a large class of examples of operators satisfying the assumptions~\eqref{eq:Tan_1},~\eqref{eq:Tan_2}, and~\eqref{eq:Tan_3}. Moreover, thanks to the results of~\cite{Tan}, we obtain several Fourier analysis spirit theorems for the convolution type operators. 

\subsubsection{Assumptions of~\cite{singular}}\label{sub:conv}
  Let $s_0$ be an even positive integer larger than $\mathbf{N}$, which will be fixed in this section. Consider a function  $K\in C^{s_0} (\mathbb R^N\setminus \{0\})$ such that 
\begin{equation}\label{eq:uni_on_annulus}\tag{A} \sup_{0<a<b<\infty} \Big| \int_{a<\|\mathbf x\|<b} K(\mathbf x)\, dw(\mathbf x)\Big|<\infty,  
\end{equation} 
\begin{equation}\label{eq:assumption1}\tag{D} 
\Big|\frac{\partial^\beta}{\partial \mathbf x^\beta} K(\mathbf x)\Big|\leq C_{\beta}\|\mathbf x\|^{-\mathbf N-|\beta|} \quad \text{for all} \ |\beta |\leq s_0, 
\end{equation}
\begin{equation}\label{eq:limitA}\tag{L} 
     \lim_{\varepsilon \to 0} \int_{\varepsilon <\|\mathbf x\|<1} K(\mathbf x)\, dw(\mathbf x)=L \text{  for some  }L \in \mathbb{C}.
     \end{equation} 
Set 
\begin{align*}
    K^{\{t\}}(\mathbf x)=K(\mathbf x)(1-\phi(t^{-1} \mathbf x)),    
\end{align*}
where $\phi$ is a  fixed  radial $C^\infty$-function    supported by the unit ball $B(0,1)$ such that $\phi (\mathbf x)=1$ for $\|\mathbf x\|<1/2$. It was proved in~\cite[Theorems 4.1 and 4.2]{singular} that under \eqref{eq:uni_on_annulus} and \eqref{eq:assumption1} the operators $f \mapsto f*K^{\{t\}}$ are bounded on $L^p(dw)$ for $1<p<\infty$ and they are of weak--type $(1,1)$ with the bounds independent of $t>0$. Further, assuming additionally \eqref{eq:limitA}, the  limit $\lim_{t\to 0} f*K^{\{t\}} (\mathbf x)$ exists and defines a bounded operator $\mathbf T$ on $L^p(dw)$ for $1<p<\infty$, which is of weak-type (1,1) as well \cite[Theorem 4.3 and  Theorem 3.7]{singular}. Moreover, in this case, 
the maximal operator 
\begin{align*}
    K^*f(\mathbf x)=\sup_{t>0} |f*K^{\{t\}}(\mathbf x)|
\end{align*}
is bounded on $L^p(dw)$ for $1<p<\infty$ and of weak-type $(1,1)$ (Theorem 5.1 of~\cite{singular}).

\subsubsection{Assumptions of~~\cite{Tan}}

In~\cite{Tan} (see also \cite{HLLW}) the following definition of Dunkl--Calder\'on--Zygmund singular integral
operators was proposed.  Let $\eta>0$. Let $\dot C_0^{\eta}(\mathbb{R}^N)$ denote the space of continuous functions $f$ with compact support satisfying
\begin{align*}
    \|f\|_{\eta}:=\sup_{\mathbf{x} \neq \mathbf{y}}\frac{|f(\mathbf{x})-f(\mathbf{y})|}{\|\mathbf{x}-\mathbf{y}\|^{\eta}}<\infty.
\end{align*}
We say that a sequence $\{f_n\}_{n \in \mathbb{N}}$ converges to $f$ in $ \dot C^{\eta}_0(\mathbb{R}^N) $, if the functions are supported in the same compact set in $\mathbb R^N$ and $\lim_{n\to\infty} \| f_n-f\|_\eta=0$. 
Let {$\dot C^{\eta}_0(\mathbb{R}^N)'$}  be its dual space endowed with weak-* topology. An operator $\mathbf{T}:\dot C^{\eta}_0(\mathbb{R}^N) \longmapsto \dot C^{\eta}_0(\mathbb{R}^N)'$  is said to be a Dunkl--Calder\'on-Zygmund singular integral operator associated with a kernel $\mathcal{K}(\mathbf{x},\mathbf{y})$ (which is not necessary the Dunkl translation of some function) if  the following estimates are satisfied: 
 for some $0 <\varepsilon \leq 1$:
\begin{equation}\label{eq:Tan_1}\tag{CZ1} 
     |\mathcal{K}(\mathbf{x},\mathbf{y})| \leq C \left(\frac{d(\mathbf{x},\mathbf{y})}{\|\mathbf{x}-\mathbf{y}\|}\right)^{\varepsilon}\frac{1}{w(B(\mathbf{x},d(\mathbf{x},\mathbf{y})))}\text{ for all }\mathbf{x} \neq \mathbf{y},
     \end{equation}
    \begin{equation}\label{eq:Tan_2}\tag{CZ2} 
     |\mathcal{K}(\mathbf{x},\mathbf{y})-\mathcal{K}(\mathbf{x},\mathbf{y}')| \leq C \left(\frac{\|\mathbf{y}-\mathbf{y}'\|}{\|\mathbf{x}-\mathbf{y}\|}\right)^{\varepsilon}\frac{1}{w(B(\mathbf{x},d(\mathbf{x},\mathbf{y})))} \text{ for all }\|\mathbf{y}-\mathbf{y}'\|<\frac{d(\mathbf{x},\mathbf{y})}{2},
     \end{equation} 
    \begin{equation}\label{eq:Tan_3}\tag{CZ3} 
     |\mathcal{K}(\mathbf{x},\mathbf{y})-\mathcal{K}(\mathbf{x}',\mathbf{y})| \leq C \left(\frac{\|\mathbf{x}-\mathbf{x}'\|}{\|\mathbf{x}-\mathbf{y}\|}\right)^{\varepsilon}\frac{1}{w(B(\mathbf{x},d(\mathbf{x},\mathbf{y})))} \text{ for all }\|\mathbf{x}-\mathbf{x}'\|<\frac{d(\mathbf{x},\mathbf{y})}{2}, 
     \end{equation} 
     and, furthermore,
     \begin{equation}\label{eq:kernelT}
         \langle \mathbf{T}f,g\rangle=\int_{\mathbb{R}^N}\int_{\mathbb{R}^N}\mathcal{K}(\mathbf{x},\mathbf{y})f(\mathbf{x})g(\mathbf{y})\,dw(\mathbf{x})\,dw(\mathbf{y}) \text{ if }\supp f \cap \supp g =\emptyset.
     \end{equation}

  We finish this subsection by the remark that the conditions 
~\eqref{eq:Tan_1},~\eqref{eq:Tan_2}, and~\eqref{eq:Tan_3} imply the following Calder\'on-Zygmund integral bounds for $\mathcal K(\mathbf x,\mathbf y)$ on the space of homogeneous type $(\mathbb R^N,\|\mathbf x-\mathbf y\|,dw)$ (see \cite{Tan}) : there is a constant $A>0$ such that  for all $r>0$ one has 
\begin{equation}\label{eq:Cald_Zyg1}
    \int_{r <\|\mathbf x-\mathbf y\|<2r }  (|\mathcal K(\mathbf x,\mathbf y)|+|\mathcal K(\mathbf y,\mathbf x)|)\, dw(\mathbf x) \leq A,
\end{equation}
\begin{equation}\label{eq:Cald-Zyg2}
    \int_{\| \mathbf y_0-\mathbf x\|>2r} (|\mathcal K(\mathbf x,\mathbf y)-\mathcal K(\mathbf x,\mathbf y_0)|+
 |\mathcal K(\mathbf y,\mathbf x)-\mathcal K(\mathbf y_0,\mathbf x)|)
    \, dw(\mathbf x)\leq A \quad \text{whenever } \ \mathbf y\in B(\mathbf y_0,r). 
\end{equation}

\subsubsection{Assumptions~\eqref{eq:Tan_1},~\eqref{eq:Tan_2}, and~\eqref{eq:Tan_3} for  convolution kernels}\label{sub:unify}

\begin{theorem}\label{teo:is_CZ}
Assume that  a kernel  $K\in C^{s_0} (\mathbb R^N\setminus \{0\})$ satisfies~\eqref{eq:assumption1} for a certain even integer $s_0>\mathbf N$. Then the kernel defined by 
\begin{equation}\label{eq:kernelCalK}
    \mathbf{K}(\mathbf{x},\mathbf{y})=\lim_{t \to 0}\tau_{\mathbf{x}}K^{\{t\}}(-\mathbf{y})=\lim_{t \to 0}K^{\{t\}}(\mathbf{x},\mathbf{y})
\end{equation}
for $\mathbf{x},\mathbf{y} \in \mathbb{R}^N$, $\mathbf{x} \neq \mathbf{y}$, satisfies the assumptions~\eqref{eq:Tan_1},~\eqref{eq:Tan_2}, and~\eqref{eq:Tan_3} with some $0<\varepsilon<\min (1,s_0-\mathbf{N})$. Moreover, if additionally~\eqref{eq:uni_on_annulus} and~\eqref{eq:limitA} are satisfied, then $\mathbf K(\mathbf x,\mathbf y)$ is a kernel associated with the Dunkl-Calder\'on--Zygmund operator $\mathbf T$. 
\end{theorem}

\begin{proof}
Let $0<\varepsilon<\min(1,s_0-\mathbf{N})$. For any $t>0$ let us denote
\begin{align*}
    K^{\{t/2,t\}}:=K^{\{t/2\}}-K^{\{t\}}.
\end{align*}
Then $K^{\{t/2,t\}}$ is $C^{s_0}(\mathbb R^N)$-function supported by $B(0,t) \setminus B(0,t/4)$ (cf.~\cite[(3.1)]{singular}), hence $\mathcal FK^{\{t/2,t\}}\in L^{1}(dw)$.  Firstly, let us consider $K^{\{t/2,t\}}$ for $t=1$. By Theorem~\ref{teo:main_xy_single_f} applied with $s=s_0$, $\varepsilon_1=\varepsilon$, and assumption~\eqref{eq:assumption1} there is a constant ${\widetilde{C}}>0$ such that
\begin{equation}
    |K^{\{1/2,1\}}(\mathbf{x},\mathbf{y})| \leq {\widetilde{C}}(1+\|\mathbf{x}-\mathbf{y}\|)^{-1}w(B(\mathbf{x},1))^{-1/2}w(B(\mathbf{y},1))^{-1/2},
\end{equation}
\begin{equation}
\begin{split}
    &|K^{\{1/2,1\}}(\mathbf{x},\mathbf{y})-K^{\{1/2,1\}}(\mathbf{x},\mathbf{y}')| \\&\leq {\widetilde{C}}\|\mathbf{y}-\mathbf{y}'\|^{\varepsilon}(1+\|\mathbf{x}-\mathbf{y}\|)^{-\varepsilon}w(B(\mathbf{x},1))^{-1/2}\left(w(B(\mathbf{y},1))^{-1/2}+w(B(\mathbf{y}',1))^{-1/2}\right)
\end{split}
\end{equation}
for all $\mathbf{x},\mathbf{y},\mathbf{y}' \in \mathbb{R}$. For the other $t>0$, note that $K_t(\mathbf x)=t^{-{\mathbf N}}K(\mathbf x/t)$ satisfies the assumption~\eqref{eq:assumption1} with the same constants $C_{\beta}$ as $K$. Hence, proceeding by scaling, for all $\mathbf{x},\mathbf{y},\mathbf{y}' \in \mathbb{R}^N$ we obtain
\begin{equation}\label{eq:scaled_CZ1}
    |K^{\{t/2,t\}}(\mathbf{x},\mathbf{y})| \leq {\widetilde{C}}\left(1+\frac{\|\mathbf{x}-\mathbf{y}\|}{t}\right)^{-1}w(B(\mathbf{x},t))^{-1/2}w(B(\mathbf{y},t))^{-1/2},
\end{equation}
\begin{equation}\label{eq:scaled_CZ2}
\begin{split}
    &|K^{\{t/2,t\}}(\mathbf{x},\mathbf{y})-K^{\{t/2,t\}}(\mathbf{x},\mathbf{y}')| \\&\leq {\widetilde{C}}\frac{\|\mathbf{y}-\mathbf{y}'\|^{\varepsilon}}{t^{\varepsilon}}\left(1+\frac{\|\mathbf{x}-\mathbf{y}\|}{t}\right)^{-\varepsilon}w(B(\mathbf{x},t))^{-1/2}\left(w(B(\mathbf{y},t))^{-1/2}+w(B(\mathbf{y}',t))^{-1/2}\right).
\end{split}
\end{equation}
 {We now turn to prove that $\mathbf K(\mathbf x,\mathbf y)$ is well defined (see \eqref{eq:kernelCalK}).}  Since $\supp K^{\{t/2,t\}} \subseteq B(0,t)$, by Theorem~\ref{teo:support} concerning the support of the Dunkl translated function, we have 
\begin{equation}\label{eq:suppKn} K^{\{t/2,t\}}(\mathbf{x},\mathbf{y})=0 \quad \text{for } t<d(\mathbf{x},\mathbf{y}).
\end{equation}
For $\mathbf x,\mathbf y \in \mathbb{R}^N$ such that  $d(\mathbf{x}, \mathbf{y})>0$, let us set 
\begin{align*}
    \mathcal{K}(\mathbf{x},\mathbf{y}):=\sum_{\ell\in \mathbb Z} K^{\{2^{\ell-1}, 2^\ell\}}(\mathbf x,\mathbf y)=\sum_{2^\ell \geq d(\mathbf{x},\mathbf{y})}K^{\{2^{\ell-1},2^\ell\}}(\mathbf{x},\mathbf{y}),
\end{align*}
where the series converges absolutely. Indeed, thanks to  \eqref{eq:scaled_CZ1} and then ~\eqref{eq:growth} we have
\begin{equation}\label{eq:absolute_K}\begin{split}
    |\mathcal{K}(\mathbf{x},\mathbf{y})| &\leq \sum_{\ell \in \mathbb{Z}}|K^{\{2^{\ell-1},2^\ell\}}(\mathbf{x},\mathbf{y})|\\
    &=\sum_{2^\ell \geq \|\mathbf{x}-\mathbf{y}\|}|K^{\{2^{\ell-1},2^\ell\}}(\mathbf{x},\mathbf{y})|+\sum_{ \|\mathbf{x}-\mathbf{y}\|>2^\ell \geq d(\mathbf{x},\mathbf{y})}|K^{\{2^{\ell-1},2^\ell\}}(\mathbf{x},\mathbf{y})|\\
        &\leq C\sum_{ 2^\ell\geq \|\mathbf x-\mathbf y\|} w(B(\mathbf x,2^{\ell}))^{-1/2}  w(B(\mathbf y,2^{\ell}))^{-1/2} \\
        &\ + C\sum_{\|\mathbf x-\mathbf y\| >2^\ell\geq d(\mathbf x,\mathbf y) } w(B(\mathbf x,2^{\ell}))^{-1/2}  w(B(\mathbf y,2^{\ell}))^{-1/2} \frac{2^{\ell\varepsilon}}{\| \mathbf x-\mathbf y\|^{\varepsilon}}\\
        &\leq C' \sum_{ 2^\ell\geq \|\mathbf x-\mathbf y\|}\frac{d(\mathbf x,\mathbf y)^N}{2^{\ell N}} w(B(\mathbf x,d(\mathbf x,\mathbf y)))^{-1} \\
         &\ + C'\sum_{\|\mathbf x-\mathbf y\| >2^\ell\geq d(\mathbf x,\mathbf y) }\frac{d(\mathbf x,\mathbf y)^N}{2^{\ell N}} w(B(\mathbf x,d(\mathbf x,\mathbf y)))^{-1} \frac{2^{\ell\varepsilon}}{\| \mathbf x-\mathbf y\|^{\varepsilon}} \\&\leq C'' w(B(\mathbf x,d(\mathbf x,\mathbf y)))^{-1}\frac{d(\mathbf x,\mathbf y)^{\varepsilon}}{\| \mathbf x-\mathbf y\|^{\varepsilon}},
\end{split}\end{equation}
where we have used the fact that $dw$ is $G$-invariant and doubling (see~\eqref{eq:doubling}), so the quantities $w(B(\mathbf{x},d(\mathbf{x},\mathbf{y})))$ and $w(B(\mathbf{y},d(\mathbf{x},\mathbf{y})))$ are comparable. Since $\tau_{\mathbf x}$ is a contraction on $L^2(dw)$, we conclude that 
\begin{equation}
    K^{\{t\}}(\mathbf x,\mathbf y)=\sum_{\ell=0}^\infty K^{\{2^\ell t,2^{\ell+1}t\}} (\mathbf x,\mathbf y)
\end{equation}
for any fixed $\mathbf x \in \mathbb{R}^N$ with convergence in $L^2(dw(\mathbf y))$. 
Now, from \eqref{eq:scaled_CZ1} and \eqref{eq:suppKn} we deduce that for $t<d(\mathbf x,\mathbf y)/4$ we have 
$$ K^{\{t\}}(\mathbf x,\mathbf y)=\sum_{2^\ell>d(\mathbf x,\mathbf y)/4}K^{\{2^{\ell-1},2^\ell\}}(\mathbf x,\mathbf y)=\mathcal{K}(\mathbf x,\mathbf y),$$
hence the limit in \eqref{eq:kernelCalK} exists and $\mathcal K(\mathbf x,\mathbf y)=\mathbf{K}(\mathbf x,\mathbf y)$ for $d(\mathbf x,\mathbf y)>0$. 

We now prove that $\mathbf K(\mathbf x,\mathbf y)$ is the kernel associated with the operator $\mathbf T$.  To this end let $f,g\in L^2(dw)$ be such that $g$ 
is compactly supported and $\supp g\cap \supp f=\emptyset$. Then there is $\eta >0$ such that $\| \mathbf x-\mathbf y\|>\delta$ for $\mathbf y\in \text{supp}\, f$ and $\mathbf x\in \text{supp}\, g$. Thus, from the results stated in Subsection~\ref{sub:conv}, we have
\begin{equation}\begin{split}
    \int_{\mathbb{R}^N} ( \mathbf Tf)(\mathbf x) g(\mathbf x)\, dw(\mathbf x)&=
    \lim_{\ell\to\infty} \iint_{\| \mathbf x-\mathbf y\|>\delta}  K^{\{2^{-\ell\}}}(\mathbf x,\mathbf y)f(\mathbf y)g(\mathbf x)\, dw(\mathbf y)\, dw(\mathbf x).
    \end{split}
\end{equation}
The  functions $K^{\{2^{-\ell\}}}(\mathbf x,\mathbf y)f(\mathbf y)g(\mathbf x)\, dw(\mathbf y)\, dw(\mathbf x)$ converge pointwise to $\mathcal K(\mathbf x,\mathbf y)f(\mathbf y)g(\mathbf x)$ and are dominated by  the integrable  function 
$$ w(B(\mathbf x,d(\mathbf x,\mathbf y))^{-1}\frac{d(\mathbf x,\mathbf y)^\varepsilon}{\| \mathbf x-\mathbf y\|^\varepsilon}|f(\mathbf y)| |g(\mathbf x)|\chi_{(\delta,\infty)}(\|\mathbf x-\mathbf y\|),$$ 
since $g$ has compact support. Hence,~\eqref{eq:kernelT} holds, by the Lebesgue dominated convergence theorem.

The proof of~\eqref{eq:Tan_2} is similar but it uses~\eqref{eq:scaled_CZ2} instead of~\eqref{eq:scaled_CZ1}. Indeed, assume $\|\mathbf{y}-\mathbf{y}'\|<\frac{d(\mathbf{x},\mathbf{y})}{2}$. Then $\frac{1}{2}d(\mathbf{x},\mathbf{y}) \leq d(\mathbf{x},\mathbf{y}')$ and, by Theorem~\ref{teo:support},

\begin{align*}
    K^{\{t/2,t\}}(\mathbf{x},\mathbf{y}) = K^{\{t/2,t\}}(\mathbf{x},\mathbf{y}')=0 \text{ if }t<\frac{d(\mathbf{x},\mathbf{y})}{2}.
\end{align*}

Consequently, by ~\eqref{eq:scaled_CZ2}, 
\begin{align*}
    &|K(\mathbf{x},\mathbf{y})-K(\mathbf{x},\mathbf{y}')| \leq \sum_{\ell \in \mathbb{Z}}|K^{\{2^{\ell-1},2^\ell\}}(\mathbf{x},\mathbf{y})-K^{\{2^{\ell-1},2^\ell\}}(\mathbf{x},\mathbf{y}')|\\& \leq \sum_{ 2^\ell \geq \frac{d(\mathbf{x},\mathbf{y})}{2}}|K^{\{2^{\ell-1},2^\ell\}}(\mathbf{x},\mathbf{y})-K^{\{2^{\ell-1},2^\ell\}}(\mathbf{x},\mathbf{y}')|\\
        &\leq  C\frac{\|\mathbf{y}-\mathbf{y}'\|^{\varepsilon}}{\| \mathbf x-\mathbf y\|^{\varepsilon}}\sum_{2^\ell\geq \frac{d(\mathbf x,\mathbf y)}{2} } w(B(\mathbf x,2^{\ell}))^{-1/2}  \big(w(B(\mathbf y,2^{\ell}))^{-1/2}+w(B(\mathbf y',2^{\ell}))^{-1/2}\big) \\
        &\leq C'\frac{\|\mathbf{y}-\mathbf{y}'\|^{\varepsilon}}{\| \mathbf x-\mathbf y\|^{\varepsilon}}\sum_{2^\ell\geq \frac{d(\mathbf x,\mathbf y)}{2} }\frac{d(\mathbf x,\mathbf y)^N}{2^{\ell N}} w(B(\mathbf x,d(\mathbf x,\mathbf y)))^{-1}  \\&\leq C'' \frac{\|\mathbf{y}-\mathbf{y}'\|^{\varepsilon}}{\| \mathbf x-\mathbf y\|^{\varepsilon}}w(B(\mathbf x,d(\mathbf x,\mathbf y)))^{-1},
\end{align*}
where we have used the fact that thank to the assumption $\|\mathbf{y}-\mathbf{y}'\|<\frac{d(\mathbf{x},\mathbf{y})}{2}$ the quantities $w(B(\mathbf{x},d(\mathbf{x},\mathbf{y})))$, $w(B(\mathbf{y},d(\mathbf{x},\mathbf{y})))$, and $w(B(\mathbf{y}',d(\mathbf{x},\mathbf{y})))$ are comparable. Finally,~\eqref{eq:Tan_3} is a consequence of the fact $K(\mathbf{x},\mathbf{y})=K( -\mathbf{y}, -\mathbf{x})$.
\end{proof}

\subsection{Dunkl transform multiplier operators.} 

Our aim of this subsection is to prove that for  bounded functions $m$ the Dunkl transform multiplier operators 
$f\mapsto \mathcal F^{-1}(m(\xi)\mathcal Ff(\xi))$ 
 admit associated kernels $K(\mathbf x,\mathbf y)$ satisfying (depending on the regularity  of $m$) \eqref{eq:Tan_1}--\eqref{eq:Tan_3} or \eqref{eq:Cald_Zyg1}--\eqref{eq:Cald-Zyg2}.

\subsubsection{Multipliers - pointwise type estimates}

For an $L^1(dw)$-function $f$ we set 
$$\mathcal F^{-1}f(\mathbf x,\mathbf y)=\int_{\mathbb R^N} f(\xi)E(i\xi,\mathbf x)E(-i\xi,\mathbf y)\, dw(\xi). $$
\begin{theorem}\label{teo:mult_C_ell}
Assume $n $ is a positive integer and $0<\varepsilon \leq 1$. There is a constants $C>0$  such that for $f \in C^{n}(\mathbb{R}^N)$ such that $\supp f \subseteq B(0,4)$ and for all $\mathbf{x},\mathbf{y}, \mathbf y' \in \mathbb{R}^N$, $\|\mathbf y-\mathbf y'\|\leq \frac{d(\mathbf{x},\mathbf{y})}{2}$,  we have
\begin{equation}\label{eq:mult_f1}
\left|\mathcal{F}^{-1}f(\mathbf{x},\mathbf{y})
\right| \leq \frac{C\|f\|_{C^{n}(\mathbb{R}^N)}}{w(B(\mathbf{x},1))^{1/2}w(B(\mathbf{y},1))^{1/2}}\left(1+\|\mathbf{x}-\mathbf{y}\|\right)^{-1}\left(1+d(\mathbf{x},\mathbf{y})\right)^{-n+1},
\end{equation}
\begin{equation}\label{eq:mult_f2}
\left|\mathcal{F}^{-1}f(\mathbf{x},\mathbf{y})
-\mathcal{F}^{-1}f(\mathbf{x},\mathbf{y}')\right| \leq \frac{C\|f\|_{C^{n}(\mathbb{R}^N)}\|\mathbf y-\mathbf y'\|^{\varepsilon}}{w(B(\mathbf{x},1))^{1/2}w(B(\mathbf{y},1))^{1/2}}\left(1+\|\mathbf{x}-\mathbf{y}\|\right)^{-1}\left(1+d(\mathbf{x},\mathbf{y})\right)^{-n+1}.
\end{equation}
\end{theorem}
For the proof we need the following lemma. 
\begin{lemma}\label{lem:l_non_negative}
   Let $n$ be a non-negative integer. Then there is a constant $C_{n}>0$ such that for $f\in C^n(\mathbb{R}^N)$, $\text{supp}\, f\subseteq B(0,4)$, and $\mathbf{x},\mathbf{y} \in \mathbb{R}^N$ one has 
   \begin{equation}\label{eq:mult_simple}
       \left|\mathcal{F}^{-1}f(\mathbf{x},\mathbf{y})
\right| \leq \frac{C_n\|f\|_{C^{n}(\mathbb{R}^N)}}{w(B(\mathbf{x},1))^{1/2}w(B(\mathbf{y},1))^{1/2}}\left(1+d(\mathbf{x},\mathbf{y})\right)^{-n}.
   \end{equation}
\end{lemma}
\begin{proof}[Proof of Lemma~\ref{lem:l_non_negative}]
    The proof goes by induction. If $n =0$, then using the Cauchy-Schwarz inequality,~\eqref{eq:E_square}, and~\eqref{eq:growth} we get 
    \begin{equation}
        \begin{split}
            |\mathcal F^{-1} f(\mathbf x,\mathbf y)|&=\left| {\mathbf c}_k^{-1}\int_{B(0,4)} f(\xi)E(i\xi, \mathbf{x})E(-i\xi,\mathbf{y})\, dw(\xi)\right|\\
            &\leq {\mathbf c}_k^{-1}\| f\|_{L^\infty}  
              \left(\int_{B(0,4)}|E(i\xi,\mathbf{x})|^2\,dw(\xi)\right)^{1/2} \left(\int_{B(0,4)}|E(i\xi,-\mathbf{y})|^2\,dw(\xi)\right)^{1/2}\\
            &  \leq
    C\| f\|_{L^\infty} {w(B(\mathbf{x},1))^{-1/2}}{w(B(\mathbf{y},1))^{-1/2}}.
        \end{split}
    \end{equation}
    Now assume that the inequality \eqref{eq:mult_simple} holds for $n$. Let $f\in C^{n+1}(\mathbb{R}^N)$, $\text{supp}\, f\subseteq B(0,4)$. Then the functions $f_j=\partial_j f\in C^{n}(\mathbb{R}^N)$ and $f^{\{\alpha\}}\in C^n(\mathbb{R}^N)$ are supported in $B(0,4)$ and
    \begin{equation}\label{eq:fjfalpha_by_f}
        \|f_j\|_{C^{n}(\mathbb{R}^N)} \leq C\|f\|_{C^{n+1}(\mathbb{R}^N)} \text{ and } \|f^{\{\alpha\}}\|_{C^{n}(\mathbb{R}^N)} \leq C\|f\|_{C^{n+1}(\mathbb{R}^N)} \text{ for }j \in \{1,\ldots,N\}, \; \alpha \in R
    \end{equation}
    (see Lemma \ref{lem:diff}).  The same calculation as in the proof of Lemma \ref{propo:formula} gives 
    \begin{equation}\label{eq:simple_f_2}
        \begin{split}
            (x_j-y_j)\mathcal F^{-1}f(\mathbf x,\mathbf y)=-\mathcal F^{-1} f_j(\mathbf x,\mathbf y)-\sum_{\alpha\in R}\frac{k(\alpha)}{2} \langle \alpha,e_j\rangle \mathcal F^{-1}f^{\{\alpha\}}(\mathbf x,\sigma_{\alpha} (\mathbf y)).
        \end{split}
    \end{equation}
    Recall that by~\eqref{eq:measure} and~\eqref{eq:d} for all $\sigma \in G$ we have $w(B(\sigma (\mathbf y),1)=w(B(\mathbf y, 1))$ and $ d(\mathbf x,\sigma(\mathbf y))=d(\mathbf x,\mathbf y)$. Using \eqref{eq:simple_f_2},~\eqref{eq:fjfalpha_by_f}, and the induction hypothesis we deduce 
    \begin{equation}\label{eq:mult_f3}
        \begin{split}
           | \mathcal F^{-1}f(\mathbf x,\mathbf y)|&\leq C_{n+1}(1+\|\mathbf x-\mathbf y\|)^{-1}
            \| f\|_{C^{n+1}(\mathbb{R}^N)} {w(B(\mathbf{x},1))^{-1/2}}{w(B(\mathbf{y},1))^{-1/2}}(1+d(\mathbf x,\mathbf y))^{-n}\\
            &\leq C_{n+1}   \| f\|_{C^{n+1}(\mathbb{R}^N)} {w(B(\mathbf{x},1))^{-1/2}}{w(B(\mathbf{y},1))^{-1/2}}(1+d(\mathbf x,\mathbf y))^{-n-1}.
        \end{split}
    \end{equation}
\end{proof}

\begin{proof}[Proof of Theorem \ref{teo:mult_C_ell} ]
    We start by  proving~\eqref{eq:mult_f1} first. Let $f_j$, $f^{\{\alpha\}}$ be as in Lemma~\ref{lem:l_non_negative}. Then, by~\eqref{eq:simple_f_2},~\eqref{eq:fjfalpha_by_f}, and Lemma~\ref{lem:l_non_negative} applied to $f_j$, $f^{\{\alpha\}}$ we get
\begin{align*}
    &|\mathcal{F}^{-1}f(\mathbf{x},\mathbf{y})| \leq C(1+\|\mathbf{x}-\mathbf{y}\|)^{-1}\left(\sum_{j=1}^{N}|\mathcal F^{-1}f_j(\mathbf x,\mathbf y)|+\sum_{\alpha \in R}|\mathcal{F}^{-1}f^{\{\alpha\}}(\mathbf{x},\sigma_{\alpha}(\mathbf{y}))|\right)\\& \leq C_{n}
            \| f\|_{C^{n}(\mathbb{R}^N)}(1+\|\mathbf x-\mathbf y\|)^{-1} {w(B(\mathbf{x},1))^{-1/2}}{w(B(\mathbf{y},1))^{-1/2}}(1+d(\mathbf x,\mathbf y))^{-n+1},
\end{align*}
so~\eqref{eq:mult_f1} is proved. Now let us prove~\eqref{eq:mult_f2}. Fix $0<\varepsilon\leq 1.$ Consider $\mathbf x,\mathbf y,\mathbf y' \in \mathbb{R}^N$, $\|\mathbf y-\mathbf y'\|\leq \frac{d(\mathbf{x},\mathbf{y})}{2}$. Let $\tilde f(\xi)=f(\xi)e^{\|\xi\|^2}$. Then $\text{supp}\, \tilde f\in B(0,4)$, $\|\widetilde{f}\|_{C^n(\mathbb{R}^N)} \leq C'_n\|f\|_{C^n(\mathbb{R}^N)}$, and 
    $$\mathcal F^{-1}f(\mathbf x,\mathbf y)=\int_{\mathbb R^N} (\mathcal F^{-1}\tilde f)(\mathbf x,\mathbf z)h_1(\mathbf z,\mathbf y)\, dw(\mathbf z).$$
    Applying~\eqref{eq:mult_f1} to $\widetilde{f}$ and then~\eqref{eq:h_t_regular},  we obtain 
    \begin{equation}\label{eq:long}
        \begin{split}
            (1+& \|\mathbf x-\mathbf y\|)  (1+d(\mathbf x,\mathbf y))^{n-1}|\mathcal F^{-1} f(\mathbf x,\mathbf y)-\mathcal F^{-1} f(\mathbf x,\mathbf y')|\\
            &\leq (1+\|\mathbf x-\mathbf y\|)  (1+d(\mathbf x,\mathbf y))^{n-1}\int_{\mathbb{R}^N} |\mathcal F^{-1} \tilde f(\mathbf x,\mathbf z)| |h_1(\mathbf z,\mathbf y)-h_1(\mathbf z,\mathbf y')|\, dw(\mathbf z)\\
            &\leq\int_{\mathbb{R}^N}  (1+\|\mathbf x-\mathbf z\|)  (1+d(\mathbf x,\mathbf z))^{n-1} (1+\|\mathbf z-\mathbf y\|)  (1+d(\mathbf z,\mathbf y))^{n-1}\\
            &\ \ \ \ \ \times |\mathcal F^{-1} \tilde f(\mathbf x,\mathbf z)| |h_1(\mathbf z,\mathbf y)-h_1(\mathbf z,\mathbf y')|\, dw(\mathbf z)\\
            &\leq C\|f\|_{C^{n}(\mathbb{R}^N)}\int_{\mathbb{R}^N}  w(B(\mathbf x,1))^{-1/2}w(B(\mathbf z,1))^{-1/2}(1+\|\mathbf z-\mathbf y\|)  (1+d(\mathbf z,\mathbf y))^{n-1}\\
            &\ \ \ \times \|\mathbf y-\mathbf y'\|(h_{2}(\mathbf z,\mathbf y)+h_{2}(\mathbf z,\mathbf y'))\, dw(\mathbf z)  
        \end{split}.
    \end{equation}
    Since $\|\mathbf y-\mathbf y'\|\leq 1$, for all $\mathbf{z} \in \mathbb{R}^N$ we have
    \begin{equation}\label{eq:zy_compare}
         (1+\|\mathbf z-\mathbf y\|)(1+d(\mathbf z,\mathbf y))^{n-1}\leq C (1+\|\mathbf z-\mathbf y'\|)(1+d(\mathbf z,\mathbf y'))^{n-1}.
    \end{equation}
    It follows from the estimate on the heat kernel (see either \eqref{eq:intro_heat_2} or Theorem~\ref{teo:1}) that 
    \begin{align*}
        \int_{\mathbb{R}^N} w(B(\mathbf z,1))^{-1/2} (1+\|\mathbf z-\mathbf y\|)(1+d(\mathbf z,\mathbf y))^{n-1}h_{2} (\mathbf z,\mathbf y)dw(\mathbf z)\leq Cw(B(\mathbf y,1))^{-1/2}.
    \end{align*}
    So we conclude the desired inequality~\eqref{eq:mult_f2} from~\eqref{eq:long} and~\eqref{eq:zy_compare}, because $w(B(\mathbf y,1))\sim w(B(\mathbf y',1))$. 
\end{proof}
\begin{corollary}
Suppose that $n \in \mathbb{N}$ is the smallest integer such that $n>\mathbf{N}$ and $m \in C^{n}(\mathbb{R}^N \setminus \{0\})$ satisfies the following Mihlin--type condition: for all $\beta \in \mathbb{N}_0^{N}$, $|\beta| \leq n$ there is a constant $C_{\beta}>0$ such that
\begin{equation}\label{eq:mihlin}
    \|\xi\|^{|\beta|}|\partial^{\beta}m(\xi)| \leq C_{\beta} \text{ for all }\xi \in \mathbb{R}^N \setminus \{0\}.
\end{equation}
Then the integral kernel $K(\mathbf{x},\mathbf{y})$ of the multiplier operator $\mathcal{T}_mf=\mathcal{F}^{-1}((\mathcal{F}f)m)$ satisfies the conditions~\eqref{eq:Tan_1},~\eqref{eq:Tan_2},~\eqref{eq:Tan_3}.
\end{corollary}

\begin{proof}
Let $\phi$ be a radial $C^\infty(\mathbb R^N)$ function, $\text{supp}\, \phi\subseteq  B(0,4)\setminus B(0,1/4)$, which  forms a resolution of the identity, that is, 
\begin{equation}\label{eq:resolution}
   \sum_{\ell\in\mathbb Z} \phi(2^{-\ell}\xi)=1, \quad \xi \in \mathbb{R}^N \setminus \{0\}.  
\end{equation}
 We write
\begin{align*}
    m(\xi)=\sum_{\ell \in \mathbb{Z}}m(\xi)\phi(2^{-\ell}\xi)=:\sum_{\ell \in \mathbb{Z}}m_{\ell}(2^{-\ell}\xi),
\end{align*}
\begin{align*}
    K_\ell (\mathbf x,\mathbf y)=\tau_{-\mathbf{y}}\mathcal{F}^{-1}\left(m(\cdot)\phi(2^{-\ell}\cdot)\right)(\mathbf{x}), \ \ \widetilde{K}_{\ell}(\mathbf{x},\mathbf{y})=(\mathcal{F}^{-1}m_{\ell})(\mathbf{x},\mathbf{y}).
\end{align*}
Then $K(\mathbf{x},\mathbf{y})=\sum_{\ell \in \mathbb{Z}}K_{\ell}(\mathbf{x},\mathbf{y})$ and, by homogeneity,
\begin{align*}
    \sum_{\ell \in \mathbb{Z}}K_{\ell}(\mathbf{x},\mathbf{y})=\sum_{\ell \in \mathbb{Z}}2^{\ell\mathbf{N}}\widetilde{K}_{\ell}(2^{\ell}\mathbf{x},2^{\ell}\mathbf{y}).
\end{align*}
Let us note that the functions $m_{\ell}$ are supported by $B(0,4)$. Moreover, it follows from~\eqref{eq:mihlin} that $\sup_{\ell \in \mathbb{Z}}\|m_{\ell}\|_{C^{n}(\mathbb{R}^N)} \leq C$. Therefore, by Theorem~\ref{teo:mult_C_ell} and~\eqref{eq:growth},
\begin{equation}\label{eq:multiplier_K_l}
\begin{split}
    &|K_{\ell}(\mathbf{x},\mathbf{y})|=2^{\ell \mathbf{N}}|\widetilde{K}_{\ell}(2^{\ell}\mathbf{x},2^{\ell}\mathbf{y})| \leq C2^{\ell \mathbf{N}}\frac{(1+2^{\ell}\|\mathbf{x}-\mathbf{y}\|)^{-1}(1+2^{\ell}d(\mathbf{x},\mathbf{y}))^{-n+1}}{w(B(2^{\ell}\mathbf{x},1))^{1/2}w(B(2^{\ell}\mathbf{y},1))^{1/2}} \\&\leq C\frac{(1+2^{\ell}\|\mathbf{x}-\mathbf{y}\|)^{-1}(1+2^{\ell}d(\mathbf{x},\mathbf{y}))^{-n+1}}{w(B(\mathbf{x},2^{-\ell}))^{1/2}w(B(\mathbf{y},2^{-\ell}))^{1/2}}\\&\leq C\left(2^{N\ell}d(\mathbf{x},\mathbf{y})^{N}+2^{\mathbf{N}\ell}d(\mathbf{x},\mathbf{y})^{\mathbf{N}}\right)\frac{(1+2^{\ell}\|\mathbf{x}-\mathbf{y}\|)^{-1}(1+2^{\ell}d(\mathbf{x},\mathbf{y}))^{-n+1}}{w(B(\mathbf{x},d(\mathbf{x},\mathbf{y})))}.
\end{split}
\end{equation}
Similarly, using~\eqref{eq:mult_f2} if $\|2^{\ell}\mathbf{y}-2^{\ell}\mathbf{y}'\| \leq 1$, and~\eqref{eq:mult_f1} if $\|2^{\ell}\mathbf{y}-2^{\ell}\mathbf{y}'\| > 1$ we get 
\begin{equation}\label{eq:multiplier_K_l_lip}
\begin{split}
    &|K_{\ell}(\mathbf{x},\mathbf{y})-K_{\ell}(\mathbf{x},\mathbf{y}')|=2^{\ell \mathbf{N}}|\widetilde{K}_{\ell}(2^{\ell}\mathbf{x},2^{\ell}\mathbf{y})-\widetilde{K}_{\ell}(2^{\ell}\mathbf{x},2^{\ell}\mathbf{y}')|\\\ &\leq C\frac{\|\mathbf{y}-\mathbf{y}'\|^{\varepsilon}}{2^{-\varepsilon \ell }}\left(2^{N\ell}d(\mathbf{x},\mathbf{y})^{N}+2^{\mathbf{N}\ell}d(\mathbf{x},\mathbf{y})^{\mathbf{N}}\right)\frac{(1+2^{\ell}\|\mathbf{x}-\mathbf{y}\|)^{-1}(1+2^{\ell}d(\mathbf{x},\mathbf{y}))^{-n+1}}{w(B(\mathbf{x},d(\mathbf{x},\mathbf{y})))}\\&+C\frac{\|\mathbf{y}-\mathbf{y}'\|^{\varepsilon}}{2^{-\varepsilon \ell }}\left(2^{N\ell}d(\mathbf{x},\mathbf{y}')^{N}+2^{\mathbf{N}\ell}d(\mathbf{x},\mathbf{y}')^{\mathbf{N}}\right)\frac{(1+2^{\ell}\|\mathbf{x}-\mathbf{y}'\|)^{-1}(1+2^{\ell}d(\mathbf{x},\mathbf{y}'))^{-n+1}}{w(B(\mathbf{x},d(\mathbf{x},\mathbf{y}')))}.
\end{split}
\end{equation}
Finally,~\eqref{eq:Tan_1} follows from~\eqref{eq:multiplier_K_l}. Indeed, fix $0<\varepsilon\leq 1$, $\varepsilon <N$. Then 
\begin{equation}\label{eq:multiplier_final}
\begin{split}
    &|K(\mathbf{x},\mathbf{y})| \leq \sum_{\ell \in \mathbb{Z}, \; 2^{\ell}d(\mathbf{x},\mathbf{y}) \leq 1}|K_{\ell}(\mathbf{x},\mathbf{y})|+\sum_{\ell \in \mathbb{Z}, \; 2^{\ell}d(\mathbf{x},\mathbf{y}) > 1}|K_{\ell}(\mathbf{x},\mathbf{y})| \\&\leq \frac{C}{w(B(\mathbf{x},d(\mathbf{x},\mathbf{y})))}\left(\sum_{\ell \in \mathbb{Z}, \; 2^{\ell}d(\mathbf{x},\mathbf{y}) \leq 1}\frac{2^{\ell N}d(\mathbf{x},\mathbf{y})^{N}}{2^{\varepsilon\ell}\|\mathbf{x}-\mathbf{y}\|^{\varepsilon}}+\sum_{\ell \in \mathbb{Z}, \; 2^{\ell}d(\mathbf{x},\mathbf{y}) > 1}\frac{2^{\ell\mathbf{N}}d(\mathbf{x},\mathbf{y})^{\mathbf{N}}}{2^{\ell}\|\mathbf{x}-\mathbf{y}\|2^{(n-1)\ell}d(\mathbf{x},\mathbf{y})^{n-1}}\right)\\&\leq C\frac{d(\mathbf{x},\mathbf{y})^{\varepsilon}}{\|\mathbf{x}-\mathbf{y}\|^{\varepsilon}}\frac{1}{w(B(\mathbf{x},d(\mathbf{x},\mathbf{y})))}.
\end{split}
\end{equation}
The proof of~\eqref{eq:Tan_2} with $\varepsilon \leq n-\mathbf N$, $0<\varepsilon \leq 1$, follows the pattern presented in~\eqref{eq:multiplier_final} but it uses~\eqref{eq:multiplier_K_l_lip} instead of~\eqref{eq:multiplier_K_l}. Finally,~\eqref{eq:Tan_3} is a consequence of the fact $K(\mathbf{x},\mathbf{y})=K( -\mathbf{y}, -\mathbf{x})$.
\end{proof}

\subsubsection{Multipliers - integral type estimates}
Let $m$ be a bounded function on $\mathbb R^N$ which for a certain $s>\mathbf N$ satisfies 
\begin{equation}\label{eq:mult_sobolv}
M:=\sup_{t>0}\| \psi(\cdot )m(t \cdot )\|_{{W^{s}_2}}<\infty,
\end{equation}
where $\psi\in C^\infty (\mathbb R^N)$ is a fixed radial function $\text{supp}\, \psi \subseteq \{\xi \in \mathbb{R}^N\;:\; 1/4\leq \|\xi \| \leq 4\}$,
$\psi (\xi) = 1$ for all $\xi \in \mathbb{R}^N$ such that $1/2\leq \|\xi\|\leq 2$, and 
\begin{align*}
\| f\|_{W^{s}_2}^2:=\int_{\mathbb R^N} (1+\|\mathbf x\|)^{2s}||\hat f(\mathbf x)|^2\, d\mathbf x    
\end{align*}
denotes the classical Sobolev norm of the classical Sobolev space $W^s_2(\mathbb{R}^N,d\mathbf{x})$. It was proved in~\cite[Theorem 1.2]{DzH} that the Dunkl multiplier operator 
\begin{align*}
    \mathcal{T}_mf=\mathcal F^{-1}\{(\mathcal Ff)m\},
\end{align*}
originally defined on $L^2(dw)\cap L^p(dw)$, has a unique extension to a bounded operator on $L^p(dw)$ for $1<p<\infty$. Moreover, $\mathcal{T}_m$ is of weak-type (1,1) and bounded on the relevant Hardy space. In order to prove the results the authors considered the integral kernels (see~\cite[(5.3)]{DzH}):
\begin{equation}\label{eq:kernel_Kl}
    K_\ell (\mathbf x,\mathbf y)=\tau_{-\mathbf{y}}\mathcal{F}^{-1}\left(m(\cdot)\phi(2^{-\ell}\cdot)\right)(\mathbf{x})=\int_{\mathbb R^N} \phi (2^{-\ell} \xi)m(\xi) E(i\xi,\mathbf x)E(-i\xi, \mathbf y)\, dw(\xi),
\end{equation}
where $\phi$ is a radial $C^\infty(\mathbb R^N)$ function, $\text{supp}\, \phi\subseteq  B(0,4)\setminus B(0,1/4)$, which  forms a resolution of the identity as in~\eqref{eq:resolution} 
and showed the following estimates with respect to $d(\mathbf x,\mathbf y)$ (see~\cite[formulas (5.8), (5.10), and (5.11)]{DzH}): there are $\delta>0$ and $C>0$ such that for all $\mathbf{y},\mathbf{y}' \in \mathbb{R}^N$ we have
\begin{equation}\label{eq:bound_no_delta}
\int_{\mathbb{R}^N} |K_\ell (\mathbf x,\mathbf y)|\, dw(\mathbf x)\leq CM,
\end{equation}
\begin{equation}\label{eq:bound^delta}
\int_{\mathbb{R}^N} |K_\ell (\mathbf x,\mathbf y)|d(\mathbf x,\mathbf y)^\delta \, dw(\mathbf x)\leq C2^{-\delta \ell}M,
\end{equation}
\begin{equation}\label{eq:int_Hol}
\int_{\mathbb{R}^N} |K_\ell(\mathbf x,\mathbf y)-K_\ell(\mathbf x,\mathbf y')|\, dw(\mathbf x)\leq CM2^{\ell} \| \mathbf y-\mathbf y'\|.
\end{equation}
The estimates imply that for every ball $B=B(\mathbf x_0,r)$ one  has 
\begin{equation}
    \int_{\mathbb R^N\setminus \mathcal O(B^*) } |K_\ell(\mathbf x,\mathbf y)-K_\ell(\mathbf x,\mathbf y')|\, dw(\mathbf x)\leq CM\min \Big((2^\ell r)^{-\delta} ,2^\ell r\Big)
\end{equation}
for all $\mathbf y,\mathbf y'\in B$.  Here  $B^*=B(\mathbf x_0, 2r)$ and $\mathcal O(B^*)=\{\sigma (\mathbf x):\sigma \in G,\ \mathbf x\in B^*\}$. The bounds \eqref{eq:bound_no_delta}--\eqref{eq:int_Hol} play  crucial roles in proving the H\"ormander's multiplier theorem (\cite[Theorem 1.2]{DzH}). 

In this subsection we will prove the following proposition.

\begin{proposition}
Suppose that $m$ is as in~\cite[Theorem 1.2]{DzH}, that is, \eqref{eq:mult_sobolv} holds for a certain $s>\mathbf N$. Let $K_\ell$ be defined by~\eqref{eq:kernel_Kl}. Then the integral kernel $K(\mathbf x,\mathbf y):=\sum_{\ell \in \mathbb{Z}} K_\ell (\mathbf x,\mathbf y)$ associated with the multiplier $\mathcal{T}_m$ satisfies the  Calder\'on--Zygmund integral conditions \eqref{eq:Cald_Zyg1}  and \eqref{eq:Cald-Zyg2}.

\end{proposition}

In other words, $\mathcal T_m$ is a Calder\'on-Zygmund operator on the space of homogeneous type $(\mathbb R^N, \|\mathbf x -\mathbf y\|, dw)$. 

\begin{proof}
Fix $s_2>\mathbf N+1$ (sufficiently large) and assume that $\eta\in W^{s_2}_2(\mathbb R^N, d\mathbf x)$, $ \text{supp}\, \eta \subseteq B(0,4)$.  Then 
\begin{equation}\label{eq:Sobolev}
    \eta_j(\cdot)=\partial_j\eta(\cdot), \ \eta_\alpha (\cdot)=\frac{\eta(\cdot)-\eta(\sigma_\alpha (\cdot))}{\langle \cdot,\alpha\rangle} \in W^{s_2-1}_2(\mathbb R^N, d\mathbf{x}) 
\end{equation}
(cf. Lemma~\ref{lem:diff}). Applying the technique from the proof of Proposition~\ref{propo:formula}, for all $j \in \{1,2,\ldots,N\}$ we have 
\begin{equation}\label{eq:Leib_eta}
    i(x_j-y_j)(\mathcal F^{-1} \eta)(\mathbf x,\mathbf y)=-(\mathcal F^{-1}\eta_j)(\mathbf x,\mathbf y)-\sum_{\alpha\in R} \frac{k(\alpha)}{2}\langle \alpha,e_j \rangle(\mathcal F^{-1}\eta_\alpha) (\mathbf x,\mathbf y). 
\end{equation}
Since  $s_2-1>\mathbf N$, it follows from (5.10) of \cite{DzH} (see \eqref{eq:bound_no_delta}) that for all $\mathbf{y} \in \mathbb{R}^N$ we have 
\begin{equation}\label{eq:bound_etas}
    \int_{\mathbb{R}^N}(|\mathcal F^{-1} \eta_j (\mathbf x,\mathbf y)|+|\mathcal F^{-1}\eta_\alpha (\mathbf x,\mathbf y)|)\, dw(\mathbf x)\leq C \left(\| \eta_j\|_{W^{s_2-1}_2}+\| \eta_\alpha\|_{W^{{s_2-1}}_2} \right)\leq C'\| \eta \|_{W^{s_2}_2}.
\end{equation}
Consequently, from \eqref{eq:Leib_eta} and \eqref{eq:bound_etas} we conclude
\begin{equation}\label{eq:bound_s2}
    \int_{\mathbb{R}^N} \| \mathbf x-\mathbf y\||(\mathcal F^{-1}\eta )(\mathbf x,\mathbf y)|\, dw(\mathbf x)\leq C\| \eta\|_{W^{s_2}_2}.
\end{equation}
Further, if $s_1>\mathbf N$ and $\eta \in W^{{s_1}}_2(\mathbb R^N, d\mathbf{x})$, $\text{supp}\,\eta \subseteq B(0,4)$, then (5.10) of \cite{DzH} (see also \eqref{eq:bound_no_delta}) implies 
\begin{equation}\label{eq:bound_s1}
     \int_{\mathbb{R}^N} |(\mathcal F^{-1}\eta )(\mathbf x,\mathbf y)|\, dw(\mathbf x)\leq C\| \eta\|_{W^{s_1}_2}.
\end{equation}
Now, \eqref{eq:bound_s2} and \eqref{eq:bound_s1} together with  the interpolation argument of Mauceri and Meda \cite{MM} (see also~\cite[Proposition 5.3]{ABDH}) give that if $s>\mathbf N$, then there are constants $C>0$ and $0<\theta<1$ such that for all $\eta \in W^{s}_2(\mathbb R^N,d\mathbf{x})$ supported in $B(0,4)$, and for all $\mathbf{y} \in \mathbb{R}^N$  we have 
\begin{equation}
     \int_{\mathbb{R}^N} \| \mathbf x-\mathbf y\|^\theta|(\mathcal F^{-1}\eta )(\mathbf x,\mathbf y)|\, dw(\mathbf x)\leq C\| \eta\|_{W^{s}_2}.
\end{equation}
Hence, by scaling, for all $\ell \in \mathbb{Z}$ and $\mathbf{y} \in \mathbb{R}^N$  we have
\begin{equation}\label{eq:K_l_theta}
    \int_{\mathbb{R}^N} \|\mathbf x-\mathbf y\|^{\theta}|K_\ell(\mathbf x,\mathbf y)|\, dw(\mathbf x)\leq C M 2^{-\theta\ell}. 
\end{equation}
Consequently, 
\begin{equation}\label{eq:half_sum1}
   \sum_{\ell \in \mathbb{Z}\;:\;2^\ell \geq r^{-1}} \int_{r\leq \|\mathbf x-\mathbf y\|<2r} |K_\ell (\mathbf x,\mathbf y)|\, dw(\mathbf x)\leq C  \sum_{\ell \in \mathbb{Z}\;:\;2^\ell \geq r^{-1}}  2^{-\theta \ell} r^{-\theta}\leq A. 
\end{equation}
Further, it follows from Lemma \ref{lem:E_square}  (see Proposition 3.7 of~\cite{DzH}) that 
\begin{equation}\label{eq:K_l_est}
    |K_\ell(\mathbf x,\mathbf y)|\leq C w(B(\mathbf x, 2^{-\ell}))^{-1/2} w(B(\mathbf y,2^{-\ell}))^{-1/2}.   
\end{equation}
{By ~\eqref{eq:doubling}, $w(B(\mathbf x,2^{-\ell}))\sim w(B(\mathbf y,2^{-\ell}))$, if $\|\mathbf x-\mathbf y\|<2r\leq 2^{\ell+1}.$
So applying \eqref{eq:K_l_est} and ~\eqref{eq:growth}, we get} 
\begin{equation*}\begin{split}
   \sum_{\ell \in \mathbb{Z}\;:\;2^\ell<r^{-1}}  \int_{r\leq \|\mathbf x-\mathbf y\|<2r} |K_\ell (\mathbf x,\mathbf y)|\, dw(\mathbf x)
   &\leq C \sum_{\ell \in \mathbb{Z}\;:\;2^\ell<r^{-1}}  \frac{w(B(\mathbf y,2r))}{ w(B(\mathbf y,2^{-\ell})) }\\&\leq C \sum_{\ell \in \mathbb{Z}\;:\;2^\ell<r^{-1}}\Big( \frac{2r}{2^{-\ell}}\Big)^N\leq A.
     \end{split}
\end{equation*}
Thus \eqref{eq:Cald_Zyg1} is proved. 

In order to prove \eqref{eq:Cald-Zyg2} we observe that \eqref{eq:int_Hol} together with \eqref{eq:K_l_theta} give 
\begin{equation}\label{eq:Lip_K_l}
    \int_{\|\mathbf x-\mathbf y_0\|>2r}|K_\ell (\mathbf x,\mathbf y)-K_\ell(\mathbf x,\mathbf y')|\leq 
  C\min \Big((2^\ell r)^{-\theta} ,2^\ell r\Big)
\end{equation}
whenever $\mathbf y,\mathbf y'\in B(\mathbf y_0,r)$. Finally~\eqref{eq:Cald-Zyg2} follows from~\eqref{eq:Lip_K_l}. 
\end{proof}

\subsection{Non-positivity of Dunkl translation operators}

In this subsection, we will use Proposition~\ref{propo:formula} to prove that for any root system $R$ and a multiplicity function $k>0$ there is $\mathbf{x} \in \mathbb{R}^N$ such that $\tau_{\mathbf{x}}$ is not a positive operator (see Theorem~\ref{teo:negative} for details). If $G=\mathbb{Z}_2$, the result follows from the explicit formula for $\tau_{\mathbf{x}}$ (see~\cite{R95}). For $G$ being symmetric group, the result was proved by Thangavelu and Xu (see~\cite[Proposition 3.10]{ThangaveluXu}).

\begin{theorem}\label{teo:negative}
For any $N \in \mathbb{N}$ there is a sequence of $N$ non-negative functions $\{\varphi_j\}_{j=1}^N$, $\varphi_j \in C^{\infty}(\mathbb{R}^N)$, such that for any system of roots $R \subset \mathbb{R}^N$ and any positive multiplicity function $k$, at least one $\varphi_j$ satisfies the following property: there are $\mathbf{x},\mathbf{y} \in \mathbb{R}^N$ such that $\varphi_j(\mathbf{x},\mathbf{y})<0$.
\end{theorem}

\begin{proof}
Let $\varphi \in C^{\infty}(\mathbb{R}^N)$ be a radial function ($\varphi(\mathbf{x})=\widetilde{\varphi}(\|\mathbf{x}\|)$) supported by $B(0,1/2)$ such that $0 \leq \varphi(\mathbf{x}) \leq 1$ for all $\mathbf{x} \in \mathbb{R}^N$ and $\varphi \equiv 1$ on $B(0,1/4)$. For $1 \leq j \leq N$ we set
\begin{equation}
    \varphi_j(\mathbf{x}):=(1+x_j)\varphi(\mathbf{x}).
\end{equation}
Since $\varphi$ is supported by $B(0,1/2)$, the functions $\varphi_j$ are non-negative. Then, using~\eqref{eq:key_formula_G_invariant}, for all $\mathbf{x},\mathbf{y} \in \mathbb{R}^N$ we have
\begin{align*}
    \varphi_j(\mathbf{x},\mathbf{y})=(1+(x_j-y_j))\varphi(\mathbf{x},\mathbf{y}).
\end{align*}
Take any $\alpha \in R$ and let $1 \leq j \leq N$ be such that $\langle\alpha,e_j \rangle \neq 0$. Then, by~\eqref{reflection}, for any $\mathbf{x} \in \mathbb{R}^N$ we get
\begin{equation}\label{eq:varphi_formula}
    \varphi_j(\mathbf{x},\sigma_{\alpha}(\mathbf x))=(1+x_j-(\sigma_{\alpha}(\mathbf{x}))_j)\varphi(\mathbf{x},\sigma_{\alpha}(\mathbf{x}))=(1+\langle\alpha,e_j \rangle\langle \mathbf{x},\alpha \rangle)\varphi(\mathbf{x},\sigma_{\alpha}(\mathbf{x})).
\end{equation}
One the one hand, let us note that for all $\mathbf{x} \in \mathbb{R}^N$ we have 
\begin{equation}\label{eq:positive}
    \varphi(\mathbf{x},\sigma_{\alpha}(\mathbf{x}))>0.
\end{equation}
Indeed, thanks to~\eqref{eq:translation-radial}, the fact that $\varphi \equiv 1$ on $B(0,1/4)$, and Theorem~\ref{teo:rejeb} we get
\begin{align*}
    &\varphi(\mathbf{x},\sigma_{\alpha}(\mathbf{x}))=\int_{\mathbb{R}^N}\widetilde{\varphi}(A(\mathbf{x},\sigma_{\alpha}(\mathbf{x}),\eta))\,d\mu_{\mathbf{x}}(\eta) \geq \int_{A(\mathbf{x},\sigma_{\alpha}(\mathbf{x}),\eta) \leq \frac{1}{4}}\,d\mu_{\mathbf{x}}(\eta)\\&=\int_{\|\sigma_{\alpha}(\mathbf{x})\|^2-\langle \sigma_{\alpha}(\mathbf{x}),\eta \rangle \leq \frac{1}{32}}\,d\mu_{\mathbf{x}}(\eta)=\mu_{\mathbf{x}}\left(U\left(\sigma_{\alpha}(\mathbf x),1/32\right)\right) \geq
 C^{-1}\frac{(1/32)^{\mathbf N/2}\Lambda(\mathbf x,\sigma_{\alpha}(\mathbf x),1/32)}{w(B(\mathbf x,\sqrt{1/32}))}>0.
 \end{align*}
On the other hand, for any $\alpha \in \mathbb{R}^N$ such that $\langle\alpha,e_j \rangle \neq 0$ there is $\mathbf{x} \in \mathbb{R}^N$ such that
\begin{equation}\label{eq:negative}
  (1+\langle\alpha,e_j \rangle \langle \mathbf{x},\alpha\rangle)<0.  
\end{equation}
Consequently, for such a $\mathbf{x}$, from~\eqref{eq:varphi_formula},~\eqref{eq:positive}, and~\eqref{eq:negative}, we obtain our claim.
\end{proof}

\begin{remark}\normalfont
The result that the generalized translations do not preserve positivity of some functions can be also obtained using the generalized heat kernel and Theorem~\ref{teo:1}. To this end let us observe that here is a constant $C_1>0$ such that for all $\mathbf{x} \in \mathbb{R}^N$ we have
\begin{equation}\label{eq:h2t_get_ht}
    C_1h_{2}(\mathbf{x}) \geq (1+\|\mathbf{x}\|)h_1(\mathbf{x}),
\end{equation}
where $h_t(\mathbf{x})$ is defined in~\eqref{eq:ht_by_translation}. We now set
\begin{equation}
    \varphi_j(\mathbf{x}):=C_1h_{2}(\mathbf{x})+x_jh_1(\mathbf{x}).
\end{equation}
Then, thanks to~\eqref{eq:h2t_get_ht}, the function  $\varphi_j$ is non-negative. Further, by~\eqref{eq:key_formula_G_invariant} together with Theorem~\ref{teo:1} (recall that $d(\mathbf x,\sigma_{\alpha}(\mathbf x))=0$), we get
\begin{equation*}
\begin{split}
    \varphi_j(\mathbf{x},\sigma_{\alpha}(\mathbf{x}))&=C_1h_{2}(\mathbf{x},\sigma_{\alpha}(\mathbf{x}))+\langle\alpha,e_j \rangle \langle \mathbf{x},\alpha \rangle h_1(\mathbf{x},\sigma_{\alpha}(\mathbf{x})) \\&\leq C_2h_{1}(\mathbf{x},\sigma_{\alpha}(\mathbf{x}))+\langle\alpha,e_j \rangle \langle \mathbf{x},\alpha \rangle h_1(\mathbf{x},\sigma_{\alpha}(\mathbf{x})).
\end{split}
\end{equation*}
Finally, by~\eqref{eq:h_positive}, we have $h_1(\mathbf{x},\sigma_{\alpha}(\mathbf{x}))>0$ and (if $\langle\alpha,e_j \rangle \neq 0$) one can take $\mathbf{x} \in \mathbb{R}^N$ such that $C_2+\langle\alpha,e_j \rangle \langle \mathbf{x},\alpha\rangle<0$. Consequently,   $\varphi_j(\mathbf{x},\sigma_{\alpha}(\mathbf{x}))<0$. 
\end{remark}

\end{document}